\documentclass[11pt]{article}
\usepackage{graphicx}
\usepackage{epsfig}
\usepackage{amssymb, amsmath}
\usepackage{amsthm}
\usepackage{setspace}
\usepackage[top = 1in, bottom = 1in, left = 1.2in, right =
1.2in]{geometry}

\numberwithin{equation}{section}

\newtheorem{theorem}{Theorem}[section]

\newtheorem{claim}[theorem]{Claim}

\newtheorem{lemma}[theorem]{Lemma}

\newtheorem{proposition}[theorem]{Proposition}
\newtheorem{remark}[theorem]{Remark}

\makeatletter
\def\nrfootnote{\@ifnextchar[\@xfootnote{\stepcounter\@mpfn
\protected@xdef\@thefnmark{\thempfn}%
\@footnotetext}}

\def\blfootnote{\xdef\@thefnmark{}\@footnotetext}
\makeatother

\begin{document}

\singlespacing

\title{Nonlinear Instability Theory of Lane-Emden stars}

\author{Juhi Jang\thanks{Department of Mathematics, University of California, Riverside, Riverside CA 92521,  USA. Email: \;juhijang@math.ucr.edu }}

\date{\today}

\maketitle

\begin{abstract}
We establish a nonlinear instability of the Euler-Poisson system for polytropic gases whose adiabatic exponents take value in $6/5<\gamma<4/3$ around the Lane-Emden equilibrium star configurations. 
\end{abstract}

\blfootnote{2010 \textit{Mathematics Subject Classification.} Primary 35L72, 35L80, 35Q85, 76N15}

\blfootnote{Keywords.  Euler-Poisson equations, Lane-Emden equation, Instability theory, Vacuum boundary, Lagrangian coordinates, Weighted energy method, Hardy inequality}


\section{Introduction}

One of the simplest fundamental hydrodynamical models to describe the motion of self-gravitating Newtonian inviscid gaseous stars is the compressible Euler-Poisson system: 
\begin{equation}
 \begin{split}\label{EP}
  \partial_t\rho +\nabla\!\cdot\! (\rho \mathbf{u})&=0\\
\rho (\partial_t \mathbf{u}+\mathbf{u}\!\cdot\!\nabla \mathbf{u}) + \nabla p&=-\rho \nabla\Phi\\
\Delta \Phi &= 4\pi \rho
 \end{split}
\end{equation}
where $(t,{x})\in \mathbb{R}_{+}\times\mathbb{R}^3 $ and $\rho\,,\,\mathbf{u}$ and $p$
denote respectively the
density, velocity, and pressure of gas. $\Phi$ is the gravitational potential and it is related to the gas through the  Poisson equation. We are interested in the polytropic gases (polytropes) for which the pressure is assumed to be a power function of the density. More specifically, we consider the following equation of state: 
\begin{equation}\label{gamma}
 p=K\rho^\gamma
\end{equation}
 where $K$ is an
entropy constant and $\gamma>1$ is the adiabatic gas exponent. The values of $\gamma$
distinguish the property of stars. For instance, $\gamma=5/3$ is 
used for a monatomic gas, $\gamma=7/5$ for a diatomic gas. The smaller $\gamma$ represents the heavier molecule of the gas.

The Euler-Poisson system \eqref{EP}  under the spherically symmetric motion -- $\rho(t,x)=\rho(t,r)$,  $\mathbf{u}(t,x)=u(t,r)\tfrac{x}{r}$ where $r=|x|$ -- reads as follows
\begin{equation}
\begin{split}
\rho_t + \frac{1}{r^2}(r^2\rho u)_r&= 0,\\
\rho u_t + \rho u u_r + p_r +\frac{4\pi\rho}{r^2}\int_{0}^{r}\rho
s^2 ds &=0.\label{EPs}
\end{split}
\end{equation}
Stationary solutions $(\rho_0(r), u_0=0)$ of (\ref{EPs}) satisfy the following 
ordinary differential equation 
\begin{equation}
\frac{dp}{dr} + \frac{4\pi\rho}{r^2}\int_{0}^{r}\rho s^2 ds
=0\label{s}
\end{equation}
which can be transformed into the famous Lane-Emden equation, and its non-negative 
solutions can be characterized according to $\gamma$ as follows \cite{Ch,lin}: Letting $M(\rho)\equiv\int 4\pi s^2\rho(s) ds$ be the total
mass of a star, if $\gamma>{6}/{5}$ and any $M>0$, then
 there exists at least one compactly supported solution
 $\rho$ such that $M(\rho)=M$. For $\gamma>{4}/{3}$,
 every solution is compactly supported and unique.
If $\gamma={6}/{5}$ and any $M>0$, there is a unique
 solution $\rho$ with infinite support. 
If $1<\gamma<{6}/{5}$, there are no stationary solutions with
finite total mass. These equilibria for $6/5<\gamma<2$ are often referred to the Lane-Emden stars.

The stability problem of the Lane-Emden star configurations has been of great interest both physically and mathematically. Its physical literature can be tracked back to Chandresekhar \cite{Ch} and it has been
conjectured by astrophysicists that the steady states for
$\gamma<{4}/{3}$ are unstable.  To see why this conjecture makes sense, we recall the total energy $E$ of Euler-Poisson system \eqref{EP}: 
\begin{equation}\label{energy}
E(\rho,u)=\int \frac{1}{2}\rho |u|^2+\frac{p}{\gamma-1} dx-\frac{1}{8\pi}\int |\nabla\Phi|^2 dx.
\end{equation} 
Note that $E$ is not
positive definite, which is a major difficulty of the stability
question. The stability is related to the competition between the
kinetic energy which tries to pull the gas away from the center and the gravitational potential energy which tries to push the gas to the center. 
 The corresponding energy for the  Lane-Emden equilibria of \eqref{s} can be written  in terms of the pressure integral: 
\[
E(\rho)=\frac{4-3\gamma}{\gamma-1}\int p \,dx. 
\]
For the explicit computation, see \cite{DLTY}. We see that the energy is negative for $\gamma>4/3$ and positive for $\gamma<4/3$ and hence one may predict the stability based on the principle of minimum energy. In fact, one can check that when $\gamma<4/3$,  by constructing a scaling invariant family of steady states, the corresponding steady states are not minimizers of the energy 
functional \eqref{energy} and this indicates a 
possibility of certain kind of instability.

 The mathematical development of the stability theory of the Lane-Emden stars is rather recent. The linear stability  
was studied by Lin \cite{lin}: any stationary solution is stable when
$\gamma>{4}/{3}$ and unstable when $\gamma<{4}/{3}$.  In
accordance with the linear stability, a nonlinear stability for
$\gamma>{4}/{3}$ was established by Rein in \cite{Rein} using the variational approach based on the fact
that the steady states are minimizers of an energy functional.
For $\gamma={4}/{3}$, in \cite{DLTY}, Deng, Liu, Yang, and  Yao
showed that the energy of a steady state is zero and any small
perturbation can make the energy positive and cause part of the
system go off to infinity.  In \cite{J0}, the author proved the fully nonlinear,
dynamical instability of the steady profile for $\gamma={6}/{5}$ based on the bootstrap argument and nonlinear weighted energy estimates. For $\gamma=6/5$, the density of the Lane-Emden star decays rapidly without vanishing in any finite radius and the analysis in \cite{J0} is not directly applied to  other cases of 
${6}/{5}<\gamma<4/3$ which attain more physical, interesting
features. This is because the gas sphere for ${6}/{5}<\gamma<4/3$ contacts with vacuum continuously,  namely the density of gas vanishes at a finite radius.  
The boundary behavior of compactly
supported Lane-Emden solutions is characterized as
follows \cite{Ch, lin}:
\begin{equation}
\rho_0(r)\sim
(R-r)^{\frac{1}{\gamma-1}}\;\text{ for }\;r\sim R\label{behavior}
\end{equation}
near the boundary. The presence of this vacuum boundary makes the problem challenging and very interesting and the nonlinear stability question has been an open problem for a long time. The purpose of this article is to establish a nonlinear instability theory of compactly supported Lane-Emden steady stars for $6/5 < \gamma < 4/3$ in the presence of the vacuum boundary.

In the next section, we introduce Lagrangian coordinates to deal with the vacuum states satisfying \eqref{behavior}, derive the Lane-Emden equation in Lagrangian formulation, and state the main result of this article.

\section{Lagrangian formulation and Main result}

\subsection{Physical vacuum}

The boundary behavior given in \eqref{behavior} naturally arises from the physical system such as the Euler-Poisson system or Euler equation with damping  and it is referred to the so-called physical vacuum  in the context of compressible fluids and gases \cite{LY2}. In general, when the initial density function contains a vacuum, the vacuum boundary $\Gamma$ is defined as
\[
 \Gamma=cl\{(t,x):\rho(t,x)>0\}\cap cl\{(t,x):
\rho(t,x)=0\}. 
\]
By introducing the
sound speed $c$
$$c=\sqrt{\frac{d}{d\rho}p(\rho)}\quad (=\sqrt{K\gamma}\rho^{\frac{\gamma-1}{ 2}}
\text{ for polytropic gases}),$$
we call a vacuum boundary $\Gamma$ {physical} if
\begin{equation}\label{pvb}
-\infty < \frac{\partial c^2}{\partial n}<0
\end{equation}
in a small neighborhood of the boundary, where
$n$ is the outward unit normal to $\Gamma$. 
Despite its physical importance, even the local existence theory of
classical solutions of compressible Euler equations featuring the physical vacuum boundary was established  only recently. This is because
if the physical vacuum boundary condition \eqref{pvb} is assumed,
the classical theory of hyperbolic systems can not be directly applied. Indeed, Nishida in \cite{Nishida} suggested to consider a free boundary problem which includes this kind of
 singularity caused by vacuum, not shock wave singularity:  under the 
  physical vacuum condition \eqref{pvb}, the boundary is supposed to move with a finite normal acceleration, and it becomes part of the unknown. Hence, one needs to deal with a moving boundary. 
Very recently, the local existence theory of the physical vacuum states for three-dimensional compressible Euler equations  has been given independently  by the author and Masmoudi 
\cite{JM10} and by Coutand and Shkoller \cite{CS10} in Lagrangian coordinates with different approaches. For more detail on physical vacuum, we refer to \cite{JM1} and references therein. 

There are some existence theories available of  Euler or Euler-Poisson equations in other smooth vacuum states when $c$ is continuously differentiable; for instance, see \cite{LY1, Makino92, MU87, MUK86}. For more discussion on the existence theory concerning vacuum states, we refer to \cite{JM10, JM1}.

In accordance with the recent advancement of physical vacuum \cite{CS10,JM10}, we will study the Euler-Poisson system \eqref{EP} or \eqref{EPs} in a suitable Lagrangian coordinates and formulate the problem in such Lagrangian coordinates.

\subsection{Derivation of the system in Lagrangian  coordinates}

First we will write the Euler-Poisson system \eqref{EP} in Lagrangian coordinates. Let $\eta(t,x)$ be the position of the gas particle $x$ at time $t$ so that
\begin{equation}\label{flow}
\eta_t = \mathbf{u}(t, \eta(t,x))\text{ for }t>0 \text{ and }\eta(0,x)=\eta_0\, \text{ in } \Omega\,.
\end{equation}
Here $\eta_0$ is not necessarily the identity map but diffeomorphic to the identity map. Our choice of $\eta_0$ will depend on the initial density profile and in fact, in our setup, the identity map will correspond to the equilibrium state. 
  
The followings are the Lagrangian quantities:
\[
\begin{split}
 \mathbf{v}(t,x)&\equiv \mathbf{u}(t, \eta(t,x))\;\text{(Lagrangian velocity)}\\ f(t,x)&\equiv \rho(t,\eta(t,x))\;
 \text{(Lagrangian density)}\\
 \Psi(t,x)&\equiv \Phi(t,\eta(t,x))\;
 \text{(Lagrangian Potential)}\\
A&\equiv[D\eta]^{-1}\;\text{(inverse of deformation tensor)}\\
J&\equiv\det D\eta\;\text{(Jacobian determinant)}\\ a&\equiv
JA\;\text{(transpose of cofactor matrix)}
\end{split}
\]
We use  Einstein's summation convention and the notation $F,_k$ to denote the $k^{\text{th}}$-partial derivative of $F$.  In this subsection, we use $i,j,k,l,r,s$ to denote 1, 2, 3.
The Euler-Poisson equations \eqref{EP} read as follows:
\begin{equation}\label{Lag}
\begin{split}
 f_t+fA^j_i\mathbf{v}^i,_{j}&=0\\ f\mathbf{v}_t^i+ KA_i^k f^\gamma,_{k}&=-fA^k_i\Psi,_k\\
 A^k_i(A^l_i\Psi,_l),_k&=4\pi f
\end{split}
\end{equation}
Since
$$J_t=JA_i^jv^i,_j \,,$$
together with the equation for $f$, we find that $$fJ=\rho(0) J(0)=\rho_{in}\,\text{det} D\eta_0$$ where
$\rho_{in}$ is a given initial density function. We now choose $\eta_0$ so that 
\begin{equation}\label{eta0}
fJ=\rho_{in}\,\text{det} D\eta_0 =\bar{\rho} 
\end{equation}
where $\bar{\rho}$ is the equilibrium density profile of the Lane-Emden star given by \eqref{s} which is our object for the instability theory. See also Section \ref{2.3} for the Lane-Emden equation. 
 
By using the relation 
$A^k_i=J^{-1}a_i^k$,  we see that the system \eqref{Lag} is reduced to the following:
\begin{equation}\label{Lag2}
\begin{split}
\bar{\rho}\mathbf{v}^i_t +Ka_i^k({\bar{\rho}}^\gamma{J}^{-\gamma}),_k&=-\bar{\rho}A^k_i\Psi,_k \\
 A^k_i(A^l_i\Psi,_l),_k&=4\pi \bar{\rho}J^{-1}
\end{split}
\end{equation}
along with
\begin{equation}\label{eta}
 \eta_t^i =\mathbf{v}^i .
\end{equation}

Now let $w$ be
\begin{equation}\label{w}
w\equiv K\bar{\rho}^{\gamma-1}.
\end{equation}
Note that $\tfrac{\gamma}{\gamma-1}w$ is the equilibrium enthalpy. We prefer to work with the enthalpy $w$ to  the density $\bar{\rho}$, since $w$ behaves like a distance function near the boundary irregardless of the values of $\gamma$ under the physical vacuum condition \eqref{behavior}. See Section \ref{2.3} for the basic property of $w$. We will treat $w$ as the weight function. This viewpoint is important not only for the existence theory but also for our instability analysis.

By using the Piola
identity: $a^k_i,_k=0$,  we see that the equation \eqref{Lag2} takes the following form
\begin{equation}\label{general}
\begin{split}
w^\alpha \mathbf{v}^i_t +(w^{1+\alpha}\,A_i^k{J}^{-1/\alpha}),_k&=-w^\alpha A^k_i\Psi,_k\\ A^k_i(A^l_i\Psi,_l),_k&={4\pi}K^{-\alpha} w^\alpha J^{-1}
\end{split}
\end{equation}
where 
\begin{equation}\label{alpha}
\alpha\equiv \frac{1}{\gamma-1}\;.
\end{equation}
Here $\alpha$ has been introduced for the notational convenience. We will use both 
$\alpha$ and $\gamma$, which are related through \eqref{alpha},  in the equations and the estimates throughout the article. For instance, the range of the adiabatic exponents of our interest reads in terms of $\alpha$ as follows: 
\[
6/5<\gamma<4/3  \;\;  \Longleftrightarrow\;\; 3<\alpha<5
\]

For our purpose of the instability theory,  it is sufficient to consider the spherically symmetric motion. 
What follows next is  the derivation of the spherically symmetric  Euler-Poisson flows in Lagrangian coordinates.

  Let
\begin{equation}\label{xi}
\eta (t,x) \equiv  \xi(t,r) \, x
\end{equation}
where $r=|x|$. Since $\eta_t=\xi_t\, x = \mathbf{v}$, $v(t,r)= r\,\xi_t $. Since $\partial_i= \tfrac{x^i}{r}\partial_r$, note that $$\eta^j,_i= \xi,_i x^j+\xi x^j,_i= \xi_r \frac{x^ix^j}{r}+\xi \delta^j_i \;\; \Rightarrow \;\; D\eta = \xi I + \frac{\xi_r}{r} (x^ix^j)$$
Thus $J$ and $[D\eta]^{-1}$ can be written as
\begin{equation}
\begin{split}\label{J}
J=\xi^2(\xi+ \xi_r r) \;;\; [D\eta]^{-1}= \frac{1}{\xi} I -\frac{\xi_r}{\xi (\xi+\xi_r r)r}\,(x^i x^j)
\end{split}
\end{equation}
and hence $A^k_i$ is given by
\begin{equation}\label{A}
A^k_i= \frac{\delta^k_i}{\xi}-\frac{\xi_r \,x^kx^i}{\xi (\xi+\xi_r r)r}
\end{equation}
and the gradient $A^k_i\partial_k$ is given by
\[
A^k_i\partial_k = \frac{x^i}{r(\xi+\xi_r r)}\partial_r
\]
and the Laplacian $A^k_i\partial_k(A^l_i\partial_l)$ is given by
\[
\begin{split}
A^k_i\partial_k(A^l_i\partial_l)&= A^k_i\partial_k\left(\frac{x^i}{r(\xi+\xi_r r)}\partial_r\right)= A^k_i \delta^i_k\frac{1}{r(\xi+\xi_r r)}\partial_r+x^i A^k_i\partial_k\left(\frac{1}{r(\xi+\xi_r r)}\partial_r\right)\\
&= \left(\frac{3}{\xi}-\frac{\xi_r r}{\xi(\xi+\xi_r r)}\right)\frac{\partial_r}{r(\xi+\xi_r r)}+ \frac{r}{\xi+\xi_r r}\partial_r
\left(
\frac{\partial_r}{r(\xi+\xi_r r)}\right)\\
&=  \frac{1}{\xi+\xi_r r}\partial_r\left(
\frac{\partial_r}{\xi+\xi_r r}\right)+ \frac{2}{\xi r(\xi+\xi_r r)}\partial_r \\
&= \frac{1}{(\xi+\xi_r r)(\xi r)^2}\partial_r \left(\frac{ (\xi r)^2}{\xi+\xi_r r}\partial_r\right)
\end{split}
\]
Thus the Poisson equation in \eqref{general} for spherically symmetric flows takes the following form:
\begin{equation}\label{sPoi}
\frac{1}{(\xi+\xi_r r)(\xi r)^2}\partial_r \left(\frac{(\xi r)^2}{\xi+\xi_r r}\Psi_r\right)= 4\pi K^{-\alpha} w^\alpha J^{-1}
\end{equation}
Based on \eqref{xi}, \eqref{J}, and \eqref{A}, we see that the momentum equation in \eqref{general} for spherically  symmetric flows can be written as an equation for $\xi$:
\begin{equation}\label{xiEq}
w^\alpha \xi_{tt} +\frac{\xi^2}{r}\partial_r \left(w^{1+\alpha}[\xi^2(\xi+\xi_rr)]^{-\frac{1+\alpha}{\alpha}}\right) +
\frac{w^\alpha}{r(\xi+\xi_r r)}\Psi_r=0
\end{equation}
for $t\geq0$ and $0\leq r\leq R$ where $R$ is the radius of the Lane-Emden star. 
We remark that no boundary conditions are necessary to construct smooth solutions for \eqref{xiEq} due to the degenerate weights \cite{CS10, JM10}. 
Note that from \eqref{sPoi} the potential term can be also written as
\[
\frac{w^\alpha}{r(\xi+\xi_r r)}\Psi_r= \frac{ w^\alpha}{\xi^2\, r^3}\int_0^r \frac{4\pi}{K^\alpha} w^\alpha s^2 ds \;\Big(= \frac{ w^\alpha}{\xi^2\, r^3}\int_{B(0,r)}\bar{\rho} \, dx\Big). 
\]
This potential term is essentially lower-order with respect to differential structure at least for a short time and moreover, it has the right weight $w^\alpha$. Hence, together with the regularity of the weight $w$  in Lemma \ref{lem2.1}, the local existence of smooth solutions to \eqref{xiEq} can be easily established based on the theories developed by the author and Masmoudi 
\cite{JM10} or by Coutand and Shkoller \cite{CS10}. 
In this article, we take the existence of smooth solutions for granted and  focus on the instability analysis.

\subsection{Lane-Emden star configuration in Lagrangian formulation}\label{2.3}

In this subsection, we will identify the Lane-Emden stars satisfying \eqref{s} in our Langrangian formulation.    
The static equilibria of the Euler-Poisson system under spherical symmetry governed by \eqref{xiEq} can be found by setting $\xi\equiv 1$, since $\rho_{in}=\bar{\rho}$ and thus $J=1$. Letting $w$ be the corresponding enthalpy profile \eqref{w}, it is clear that $w$ satisfies the following ordinary differential equation 
\begin{equation}\label{EQ}
\frac{1}{r} \partial_r(w^{1+\alpha})+\frac{w^\alpha}{r^3}\int_0^r \frac{4\pi}{K^\alpha} w^\alpha s^2 ds=0
\end{equation}
or equivalently
\begin{equation}\label{L-E}
w_{rr}+\frac{2}{r}w_r +\frac{4\pi}{(1+\alpha)K^\alpha} w^{\alpha} =0
\end{equation}
The equation \eqref{L-E} is the so-called Lane-Emden equation. As mentioned in the introduction, this 
equation has been studied extensively. In particular, we recall the well-known  existence result from  \cite{Ch, lin}:  Supplemented with the following normalized boundary conditions 
\[
w(0)=1\text{ and } w_r(0)=0
\]
 for a given finite total mass $M$, there exist a ball-type solution $w$ to the Lane-Emden equation \eqref{L-E} and a finite radius $R$ when $1<\alpha<5$, equivalently $6/5<\gamma<2$ such that  
 
 \
 
 (i) $w>0$ for $0<r<R$ and   $w(R)=0$;
 
 (ii) $-\infty <w_r <0$ for $0<r<R$;
 
 (iii) $w$  satisfies the physical vacuum condition \eqref{behavior} or \eqref{pvb}. 
 
 \

It turns out that the Lane-Emden configuration $w$  enjoys better regularity. In Lemma \ref{lem2.1}, we will study the higher regularity of $w$, which will be helpful  to infer the suitable regularity of the growing mode solution. 

\subsection{Main Result}

To state the main theorem, we write the equation \eqref{xiEq} in a perturbation form around  the equilibrium state given by $\xi=1$ and $\xi_t=0$. Letting $\xi \equiv 1+ \zeta$ where $|\zeta|\ll 1$, we obtain the equation for  $\zeta$ as follows: 
\begin{equation}\label{zeta}
\begin{split}
w^\alpha \zeta_{tt} +\frac{(1+\zeta)^2}{r}\partial_r\left(w^{1+\alpha}[(1+\zeta)^2(1+\zeta+\zeta_rr)]^{-\frac{1+\alpha}{\alpha}}\right)&\\
+  \frac{ w^\alpha}{(1+\zeta)^2\, r^3}\int_0^r \tfrac{4\pi}{K^\alpha} w^\alpha s^2& ds
=0.
\end{split}
\end{equation}

Our main result establishes the fully nonlinear dynamical instability of the Lane-Emden equilibrium stars satisfying \eqref{L-E} under the Euler-Poisson dynamics \eqref{zeta}.   In the statement of the following theorem, for any given $\delta> 0$ and $2\theta>\delta$, we define 
\begin{equation}\label{T}
T^\delta \equiv \frac{1}{\sqrt{\mu_0}} \ln \frac{2\theta}{\delta}
\end{equation}
where $\sqrt{\mu_0}$ is the sharp linear growth rate.

\begin{theorem}\label{main} Assume that  $6/5<\gamma<4/3$ $(3<\alpha<5)$. Then there exist functions spaces $Z^\alpha_0$ and $Z^\alpha_1$ as well as constants $\theta>0$ and $C>0$ 
such that for any sufficiently small $\delta>0$, there exists a family of solutions $(\partial_t\zeta^\delta,\, \zeta^\delta)$ to the Euler-Poisson system \eqref{zeta} for $t\in [0, T)$ with $T>T^\delta$ so that 
\[
\left\|\left(\partial_t\zeta^\delta(0),\,\zeta^\delta(0)\right) \right\|_{Z^\alpha_1}\leq C \delta\; \text{ but }\; \sup_{0\leq t\leq T^\delta} 
\left\| \left(\partial_t\zeta^\delta(t),\,\zeta^\delta(t)\right)\right\|_{Z^\alpha_0}\geq \theta. 
\]
\end{theorem}

The precise statement of Theorem \ref{main} is given in Theorem \ref{thm} and the spaces $Z^\alpha_0$ and $Z^\alpha_1$ will be clarified in Section \ref{4}.

\begin{remark}
The escape time $T^\delta$ is determined through \eqref{T} by the linear growth rate $\sqrt{\mu_0}$. We note that the instability occurs before the possible breakdown or any collapse of strong solutions. We also remark that the instability occurs in the $Z_0^\alpha$ norm, which is indeed characterized by the instantaneous physical energy \eqref{E0}. 
\end{remark}

\begin{remark}
The above result shows that no matter how small amplitude of initial perturbed data is  taken to be, we can find a solution such that the corresponding energy escapes at a time $T^\delta$: there is no stabilization of the system. We conclude from this that all Lane-Emden steady star configurations for $6/5<\gamma<4/3$ are nonlinearly unstable. 
\end{remark}

\begin{remark} The instability given in Theorem \ref{main} is not related to the vorticity nor to the geometry  of the boundary, nor to the topological change of the domain. In fact, our result indicates that for gases of heavier molecules $6/5<\gamma<4/3$, the self-gravitation may be strong enough to attract matter to the center.  An interesting problem would be to investigate the possible gravitational collapse phenomenon rigorously. 
\end{remark}

Up to our knowledge, this work is the first rigorous result to address the nonlinear instability of the Lane-Emden stars for  $6/5<\gamma<4/3$ under the Euler-Poisson system. Our results together with the earlier works \cite{DLTY, J0, Rein} now complete a nonlinear stability and instability theory of the Euler-Poisson system for the Lane-Emden equilibrium stars, which provides a rigorous account of the longstanding conjecture by astrophysicists \cite{Ch}.  Also, we believe that the above result for the Lane-Emden stars and the methods developed in this paper will shed some light on other stability problems in the context of free boundary. 

We mention that based on strong solutions in the mass Lagrangian coordinates \cite{J}, the nonlinear instability theory for the Lane-Emden stars under the Navier-Stokes-Poisson system has been recently established by the author and Tice in \cite{JT}: it was shown that there's no stabilization of the system for $6/5<\gamma<4/3$ in the presence of viscosity around the Lane-Emden equilibria. On the other hand, away from the Lane-Emden equilibria,  there are interesting works on the stability theory of the Euler-Poisson system \eqref{EP} by Luo and Smoller \cite{luos2, luos}, where  a nonlinear dynamical stability of rotating stars for  $\gamma>4/3$ was established based on a variational approach by the energy-Casimir technique and  a finite time stability of the entropy weak solutions was proven. 

The proof of Theorem \ref{main} is based on a bootstrap argument from a linear instability to a nonlinear dynamical instability. Passing from a linearized instability to nonlinear instability demands much effort in the PDE context since the spectrum of the linear part is fairly complicated and the unboundedness of the nonlinear part usually yields a loss in derivatives. The key is to identify the right function spaces where the nonlinear instability would occur and this requires a careful nonlinear analysis.  We develop nonlinear energy estimates for the whole system so that: first, the nonlinear estimates can be closed; and second, their interplay with the linear analysis can complete the argument. For this particular problem, due to the degeneracy caused by the physical vacuum, the novel nonlinear weighted energy estimates together with Hardy inequalities are employed.

A key step is to derive Proposition \ref{prop6.1}, which is an important estimate for the bootstrap argument. There are a few new ingredients. In order to avoid the coordinate singularity at the origin due to the consideration of spherically symmetric flows as well as the vacuum singularity due to the physical vacuum boundary, we will apply the time differentiations to obtain the nonlinear energy inequalities systematically and then recover all the spatial derivatives from the equation by exploiting the elliptic estimates. A pivotal tool to justify such a process is the Hardy inequality; in particular, we will make use of the localized Hardy inequalities near the origin and near the boundary and such Hardy inequalities allow us to overcome the coordinate singularity near the origin as well as the vacuum singularity.  In addition to the homogeneous weighted energy $\mathcal{E}$, we further introduce the nonlinear energy $\mathfrak{E}$. This is because $\mathcal{E}$ works efficiently for the bootstrap argument  from the linear analysis  to the nonlinear solutions, while  the nonlinear energy $\mathfrak{E}$ captures better the intriguing structure of nonlinear, degenerate pressure gradient term. We will derive the energy inequality for  $\mathfrak{E}$ having the form of 
\[
\frac{d}{dt}\mathfrak{E}\leq \beta \mathfrak{E} + \text{ lower-order derivative terms}
\]
where $\beta$ is no greater than the sharp linear growth rate $\sqrt{\mu_0}$.  To do so, we identify the quantities whose $L^\infty$ bounds can be controlled within the total energy and use them crucially in controlling the nonlinear terms. With the equivalence of two energies $\mathcal{E}$ and  $\mathfrak{E}$, we can close the argument.

The rest of the paper proceeds as follows. In Section \ref{3}, we construct the largest growing mode solution to the linearized Euler-Poisson system, study the regularity of such growing mode profile, and show that our growing mode solution grows at the largest possible rate.   In Section \ref{4}, the instant temporal energy and the total energy are introduced and we establish the equivalence of two energy norms  via various Hardy inequalities. 
 In Section \ref{5}, we derive high-order nonlinear weighted energy inequalities for the nonlinear energy norms and show that the nonlinear energy is equivalent to the homogeneous total energy. Based on the linear growth and the nonlinear estimates, we then prove the bootstrap argument and Theorem \ref{main} in Section \ref{6}.

Throughout the paper, we will write $C$ for a generic constant unless the distinction between constants is vital. And we use $\mathfrak{A}\precsim \mathfrak{B}$ to denote that $\mathfrak{A}\leq C\mathfrak{B}$ for a constant $C>0$. 


\section{The largest growing mode and Linear instability}\label{3}

In this section, we study the linearized Euler-Poisson system around the Lane-Emden stars and show the linear instability of the Lane-Emden stars whose indices are $6/5<\gamma<4/3$ ($3<\alpha<5$) in the sense of Lin's criterion \cite{lin}. We will first derive the linearized equation of \eqref{zeta}. Notice that by Taylor's theorem, for sufficiently small $\zeta$, the nonlinear pressure term in \eqref{zeta}  can be written as  
\begin{equation}\label{exp}
\left[(1+\zeta)^2(1+\zeta+\zeta_rr)\right]^{-\frac{1+\alpha}{\alpha}} = 1 -\frac{1+\alpha}{\alpha} (3\zeta+\zeta_r r) +h(\zeta,\zeta_rr )
\end{equation}
where $h$ is a smooth function in both arguments $\zeta$ and $\zeta_rr$ and $h(\zeta,\zeta_rr ) = O(|\zeta|^2+|\zeta_r r|^2) $, 
and thus the linearized equation of \eqref{zeta} reads as 
\begin{equation}\label{zetaL}
w^\alpha \zeta_{tt}+ \frac{4[w^{1+\alpha}]_r}{r} \zeta -\frac{1+\alpha}{\alpha\, r}
\left[w^{1+\alpha}(3\zeta+\zeta_r r)\right]_r=0
\end{equation}
where we have used  \eqref{EQ}.

To find a growing mode, we seek a solution to \eqref{zetaL} of the form $\zeta(t,r)=e^{\lambda t}\phi(r)$. Then $\phi$ would satisfy
\begin{equation}\label{lambda1}
\lambda^2 w^\alpha\phi + \underbrace{\frac{4[w^{1+\alpha}]'}{r} \phi -\frac{1+\alpha}{\alpha\, r}
[w^{1+\alpha}(3\phi+\phi' r )]'}_{(\ast)}=0
\end{equation}
where we denote $'=\frac{d}{dr}$. The $(\ast)$ in \eqref{lambda1} can be rewritten as
\[
(\ast)=(4-3\frac{1+\alpha}{\alpha} )\frac{[w^{1+\alpha}]'}{r} \phi - \frac{1+\alpha}{\alpha \,r^4}[w^{1+\alpha}\phi' r^4]'. 
\]
Multiply  \eqref{lambda1} by $r^4$ to get
\begin{equation}\label{lambda2}
\lambda^2  { w^\alpha r^4}\phi =\frac{1+\alpha}{\alpha} [w^{1+\alpha} r^4\phi']' - (4-3\frac{1+\alpha}{\alpha} ){[w^{1+\alpha}]'}\, r^3 \phi
\end{equation}
We denote the right-hand side of \eqref{lambda2} by $L\phi$:
\begin{equation}\label{L}
L\phi\equiv \frac{1+\alpha}{\alpha} [w^{1+\alpha} r^4\phi']' - (4-3\frac{1+\alpha}{\alpha} ){[w^{1+\alpha}]'}\, r^3 \phi
\end{equation}
Then $L$ is self-adjoint and hence $\lambda^2$ is real. 

We recall that in \cite{lin}, the  stability criterion was introduced based on the eigenvalues: $w^\alpha \, (\sim \bar{\rho})$ is called neutrally stable if $\lambda^2<0$ for all eigenvalues $\lambda$ and unstable if $\lambda_1^2>0$ for some eigenvalue $\lambda_1$; and that it was shown that $w^\alpha \, (\sim \bar{\rho})$ is unstable for any $3<\alpha<5$ ($6/5<\gamma<4/3$) in mass Lagrangian framework.   In the next subsection, we verify Lin's linear instability result  in our Lagrangian framework by showing the existence of the positive largest eigenvalue for $3<\alpha<5$. 

\subsection{The existence of the largest growing mode}

We start from $L\phi = \mu  { w^\alpha r^4}\phi $ where $L$ is defined in \eqref{L} and $\mu\equiv \lambda^2$. It is well-known that the largest eigenvalue $\mu_0$ of this equation can be found by the variational formula 
\[
\mu_0=\sup \left\{ \frac{Q(\phi)}{I(\phi)}: Q(\phi)<\infty, I(\phi)<\infty\right\}
\] 
where 
\[
\begin{split}
Q(\phi)&\equiv (L\phi,\phi)= -\frac{1+\alpha}{\alpha} \int_0^R w^{1+\alpha} r^4(\phi')^2 dr  
- (4-3\frac{1+\alpha}{\alpha} )\int_0^R{[w^{1+\alpha}]'}\, r^3 \phi^2 dr, \\
I(\phi)&\equiv ( { w^\alpha r^4}\phi ,\phi) =\int_0^R w^\alpha r^4 \phi^2 dr.
\end{split}
\]
Here we have used $(\,, \,)$ to denote the standard one-dimensional inner product on $(0,R)$. 

Define a norm for any $\phi\in C^\infty([0,R])$, 
\[
\|\phi\|^2\equiv \int_0^R w^{1+\alpha} r^4(\phi')^2 dr  -\int_0^R{[w^{1+\alpha}]'}\, r^3 \phi^2 dr
+ \int_0^R w^\alpha r^4 \phi^2 dr. 
\]
Let $H=\overline{C^\infty([0,R])}$ in the above norm. Since $\frac{Q(\phi)}{I(\phi)}=\frac{Q(c\phi)}{I(c\phi)}$ for any nonzero constant $c$, the variational problem can be restated as finding the maximum $\mu_0$ of $Q(\phi)$ on $H$ under the normalization $I(\phi)=1$. The next proposition shows that the supremum $\mu_0$ can be achieved on $H$. 

\begin{proposition}\label{prop3.1}
There exists a $\phi_0\in H$ such that $I(\phi_0)=1$ and $Q(\phi_0)=\mu_0>0$. 
\end{proposition}

\begin{proof} The first part in $Q$ is negative and the second part in $Q$ is positive because $w'<0$ and $4-3\frac{1+\alpha}{\alpha}>0$ since $\alpha>3$. Thus we first want to make sure that $\mu_0>0$. This can be easily checked by taking $\phi=1$: 
\[
\mu_0\geq \frac{Q(1)}{I(1)}= \frac{ 0 -(4-3\frac{1+\alpha}{\alpha} )\int_0^R{[w^{1+\alpha}]'}\, r^3 dr}{\int_0^R w^\alpha r^4 dr}>0 .
\]
Note that the positive part (the second term) of $Q$ is uniformly bounded by $I$ because from \eqref{EQ}, we have  that $-\frac1{rw^\alpha}[w^{1+\alpha}]' =\frac{1}{r^3}\int_0^r\frac{4\pi}{K^\alpha} w^\alpha s^2 ds$ is bounded. 
This implies that $\mu_0$ is finite. To verify that $\mu_0$ is attained on $H$, we will need a compactness argument. Let $\phi_n$ be a maximizing sequence so that 
\[
Q(\phi_n) \nearrow \mu_0\; \text{ as }\; n\rightarrow \infty \;\text{ and }\; I(\phi_n)=1 \;\text{ for all }n. 
\]
Let $\phi_0$ be its weak limit. By the lower semicontinuity of weak convergence, we have 
\[
\begin{split}
&\liminf \int_0^R  w^{1+\alpha} r^4(\phi'_n)^2 dr  \geq \int_0^R w^{1+\alpha} r^4(\phi_0')^2 dr \\
&\liminf \int_0^R  w^{\alpha} r^4(\phi_n)^2 dr  \geq \int_0^R w^{\alpha} r^4(\phi_0)^2 dr \\
\end{split}
\]
The next step  is to prove the following compactness result. 

\begin{claim} The mass is not lost in the limit: $I(\phi_0)=1$. Moreover, we have the compactness result for the positive part of $Q$: there exists a subsequence $\{\phi_{n_k}\}$ of $\{\phi_n\}$ such that 
\begin{equation}\label{compact}
\int_0^R{-[w^{1+\alpha}]'}\, r^3 \phi_{n_k}^2 dr \rightarrow \int_0^R{-[w^{1+\alpha}]'}\, r^3 \phi_0^2 dr
\;\text{ as }\; n\rightarrow \infty. 
\end{equation}
\end{claim}

Since the above claim implies $Q(\phi_0)=\mu_0$ and $I(\phi_0)=1$, the conclusion of the proposition follows from the claim. 

The first claim can be shown by a standard scaling argument. Suppose $I(\phi_0)= \beta^2<1$. Then we would have 
\[
I(\frac{\phi_0}{\beta})=1\; \text{ and } \;Q(\frac{\phi_0}{\beta})=\frac{\mu_0}{\beta^2} >\mu_0
\]
which contradicts to the definition of $\mu_0$. Hence,  $I(\phi_0)=1$. Next, in order to see \eqref{compact}, we 
rewrite it by using \eqref{EQ}
\[
\int_0^R{-[w^{1+\alpha}]'}\, r^3 \phi_{n}^2 dr = \int_0^R  \left(\frac{1}{r^3}\int_0^r\frac{4\pi}{K^\alpha} w^\alpha s^2 ds\right) w^\alpha  r^4 \phi_n^2dr
\]
Then since $\frac{1}{r^3}\int_0^r\frac{4\pi}{K^\alpha} w^\alpha s^2 ds$ is uniformly bounded from below and above on $(0,R)$, the positive part of $Q$ is bounded from below and above by $I$: 
\[
\int_0^R{-[w^{1+\alpha}]'}\, r^3 \phi_{n}^2 dr = \int_0^R  \left(\frac{1}{r^3}\int_0^r\frac{4\pi}{K^\alpha} w^\alpha s^2 ds\right) w^\alpha  r^4 \phi_n^2dr \sim \int_0^R w^\alpha r^4 \phi_n^2 dr =I(\phi_n)
\]
Now since $I(\phi_n) \rightarrow I(\phi_0)$ with both being one, \eqref{compact} follows. 
\end{proof}

It is straightforward to check that $\phi_0$ is in fact an eigenfunction corresponding to $\mu_0$ for \eqref{L}: $L\phi_0=\mu_0 w^\alpha r^4\phi_0$ and  therefore, $\zeta(t,r)=e^{\sqrt{\mu_0} t}\phi_0(r)$ is a largest growing mode solution to the linearized Euler-Poisson equation \eqref{zetaL}.

\subsection{The behavior of $\phi_0$ near the origin and near the boundary}\label{better}

We record the second-order ordinary differential equation which the growing mode $\phi_0$ satisfies: 
\begin{equation}\label{phi0}
\mu_0 w^\alpha r^4\phi_0 = \frac{1+\alpha}{\alpha} [w^{1+\alpha} r^4\phi_0']' - (4-3\frac{1+\alpha}{\alpha} ){[w^{1+\alpha}]'}\, r^3 \phi_0
\end{equation}
We can further deduce the regularity of $\phi_0$ from \eqref{phi0} based on the  behavior of the Lane-Emden solution ${w}$. To do so, we first derive the higher regularity and integrability of $w$. 

\begin{lemma}[Regularity of $w$] \label{lem2.1} Let $1<\alpha<5$ be given. Let $w$ be a ball-type solution to the Lane-Emden equation \eqref{L-E}. Then 
\begin{enumerate}
\item  $w$ is analytic near the origin. Moreover, 
\[
w(r)=1- b r^2+O(r^4), \;\;r\sim 0
\]
for some positive constant $b>0$. Also, $(\partial_r^{2k+1}w)(0)=0$ for any non-negative integer $k\geq0$. 

\item $\partial_r^iw$ is uniformly bounded on $(0,R)$ for each $0\leq i\leq \alpha+2$ and also $w^{\frac{k-1}{2}}\partial_r^{k+1}w$ is uniformly bounded on $(0,R)$ for each $1\leq k\leq 2\alpha+1$. 
In addition, $w$ enjoys the following integral regularity. 
\[
\int_0^R w^{\alpha+j} r^4|\partial_r^{j+1} w|^2 dr <\infty
\]
for each $0\leq j< 3\alpha +3$.
\end{enumerate}
\end{lemma}

The proof is elementary, but it has not been found in the literature, so we will provide a brief   argument. 

\begin{proof} 
We start with the first item. Write the Lane-Emden equation \eqref{L-E} as 
\begin{equation}\label{22}
rw_{rr}+2w_r+ \mathfrak{c} r w^{\alpha}=0
\end{equation}
where $\mathfrak{c}=\frac{4\pi}{(1+\alpha)K^\alpha}>0$ is a constant. After dividing \eqref{22} by $r$ and taking the limit $r\rightarrow 0^+$, we immediately deduce that $3w_{rr}(0)+\mathfrak{c}=0 $.  Take $\partial_r$ of \eqref{22}: 
\begin{equation}\label{221}
rw_{rrr}+ 3w_{rr}+\mathfrak{c} ( \alpha rw^{\alpha-1} w_r + w^{\alpha})=0
\end{equation}
By dividing it by $r$ and using the fact that $3w_{rr}(0)+\mathfrak{c}=0 $, 
\[
w_{rrr}+ \frac{3w_{rr}-3w_{rr}(0)}{r}+\mathfrak{c} ( \alpha w^{\alpha-1} w_r + \frac{w^{\alpha}-w^{\alpha}(0)}{r})=0
\]
and by passing to the limit $r\rightarrow 0^+$, we see that $4 w_{rrr}(0)=0$. 
Take one more derivative 
\begin{equation}\label{222}
rw_{rrrr}+ 4w_{rrr}+\mathfrak{c} (\underbrace{ \alpha rw^{\alpha-1} w_{rr} + 2\alpha w^{\alpha-1} w_r}_{=-\alpha \mathfrak{c} r w^{2\alpha-1}}
+\alpha(\alpha-1)rw^{\alpha-2} (w_r)^2)=0
\end{equation}
and we deduce that $5w_{rrrr}(0)-\alpha\mathfrak{c}^2=0$. More generally, we obtain the following. 

\begin{claim}
For any $k\geq 1$, the solutions to \eqref{22} satisfy the following
\begin{equation}\label{23}
r\partial_r^{2k}w+2k \partial_r^{2k-1} w +\mathfrak{c}\partial_r^{2k-2}(rw^\alpha) =0
\end{equation}
for $0<r<R$ such that 

 (a) $\partial_r^{2k-2}(rw^\alpha)(0)=0$ and $\partial_r^{2k}(rw^\alpha)(0)=0$

(b) $ (\partial_r^{2k-1} w)(0)=0$. 
\end{claim}

We will prove the claim by induction on $k$. First note that $(a)$ and $(b)$ are satisfied for $k=1$ and $k=2$ by the above computations. Now let us assume that it is true for $k=n$. Then by dividing \eqref{23} when $k=n$ by $r$ and taking the limit $r\rightarrow 0^+$, we deduce that 
\begin{equation}\label{24}
(2n+1) (\partial_r^{2n}w)(0)+ l_n=0
\end{equation}
where $l_n=\lim_{r\rightarrow 0^+} [ \mathfrak{c}\partial_r^{2n-2}(rw^\alpha)/r]$. 
Now we formally take $\partial_r$ of \eqref{23}: 
\[
r\partial_r^{2n+1}w+(2n+1) \partial_r^{2n} w +   \mathfrak{c}\partial_r^{2n-1}(rw^\alpha)  =0
\] 
By using \eqref{24}, we see that 
\[
\partial_r^{2n+1}w+(2n+1)\frac{ \partial_r^{2n} w -   (\partial_r^{2n}w)(0) }{r} +  \frac{ \mathfrak{c}\partial_r^{2n-1}(rw^\alpha)  -l_n}{r}=0
\]
Take the limit to get 
\[
(2n+2) (\partial_r^{2n+1}w)(0) +\mathfrak{c} \partial_r^{2n}(rw^\alpha)(0) =0
\]
Since $\partial_r^{2n}(rw^\alpha)(0)=0$, we deduce that $ (\partial_r^{2n+1}w)(0)=0$. Also,  the direct computation shows that $\partial_r^{2n+2}(rw^\alpha)(0)=0$. And thus $(a)$ and $(b)$ hold for $k=n+1$. Hence, we deduce that $w$ is infinitely differentiable at $r=0$ and the above estimation on the derivatives implies the analyticity of $w$ near the origin. This finishes the above claim as well as the first part of the lemma. 

In order to prove the point 2 of the lemma, it suffices to look at $w$ near $r=R$ since $w$ is regular near the origin. Recall that $w_r= O(1)$ for $r\sim R$. 
From \eqref{22}, we see that  $w_{rr}= O(1)+O(w^\alpha)$ for $r\sim R$. Similarly, from \eqref{221} and \eqref{222}, $w_{rrr}=O(1)+O(w^{\alpha-1})$ and $w_{rrrr}=O(1)+ O(w^{\alpha-2})$. Inductively, we obtain 
$|\partial_r^{j} w|=O(1)+O(w^{\alpha-j+2})$. Note that $\partial_r^jw$ is not bounded any longer if $j>\alpha +2$.  However, the integral 
\[
\int_0^R w^{\alpha+j}r^4|\partial_r^{j+1} w|^2 dr \leq C+ C\underbrace{\int_0^R w^{\alpha+j}w^{2(\alpha-j+1)} dr}_{ = \int w^{3\alpha -j+2}dr}
\]
stays bounded as long as $j<3\alpha+3$. On the other hand, if we consider $w^{\frac{k-1}{2}}\partial_r^{k+1}w$, then we see that $|w^{\frac{k-1}{2}}\partial_r^{k+1}w|=O(1)+ O(w^{\alpha-k+1+\frac{k-1}{2}} )$, which is uniformly bounded if $\alpha-k+1+\frac{k-1}{2}\geq0\Rightarrow 1\leq k\leq 2\alpha+1$.  
This completes the proof of the lemma. 
\end{proof}

The following lemma concerns the behavior of $\phi_0$ near the origin. 

\begin{lemma}\label{lem3.3} Let $\phi_0\in H$ be the solution to \eqref{phi0}. 
$\phi_0$ is analytic at $r=0$ and moreover, $\phi_0=a+O(r)$ around the origin where $a$ is a constant.  
\end{lemma}

\begin{proof} We rewrite \eqref{phi0} in the following form 
\[
\phi_0''+\left[\frac{4}{r} + \frac{(1+\alpha)w' }{w}\right]\phi_0' -\left [\alpha(4-3\frac{1+\alpha}{\alpha})\frac{w'}{rw} +\frac{\alpha \mu_0}{1+\alpha}\frac{1}{w}     \right] \phi_0=0 
\]
Let $P(r)$ and $Q(r)$ be coefficients of $\phi_0'$ and $\phi_0$ respectively. Then by Lemma \ref{lem2.1}, $rP(r)$ and $r^2Q(r)$ are analytic at $r = 0$, which implies $r=0$ is a regular singular point. Therefore, we may employ the Frobenius method \cite{B} (p.215).  Let $p_0$ and $q_0$ be the zeroth order term of $rP(r)$ and $r^2Q(r)$  respectively. It is easy to see that $p_0=4$ and $q_0=0$. The indicial equation $r(r-1)+p_0r+q_0=0$ has two roots $r_1=0$ and $r_2=-3$. Hence, by the Frobenius theorem, $\phi_0$ has a power series representation 
of $y_1(r)=\sum_{n=0}^\infty a_n r^n$ or $y_2(r)= y_1(r)\ln r+ r^{-3} \sum_{n=0}^\infty b_n r^n$. However, 
$y_2(r)$ is impossible since $\phi_0\in H$ and hence $\phi_0=\sum_{n=0}^\infty a_n r^n$. 
\end{proof}

Next, we show that $\phi_0$ enjoys the better integrability near the boundary than the integrability dictated by $H$ in Proposition \ref{prop3.1}. This will be done through the Hardy inequality and a bootstrapping argument through the equation. In addition, we show that $\phi_0$ indeed belongs to some weighted Sobolev spaces. 

\begin{lemma}\label{lem3.4} Let $\phi_0$ be the solution to \eqref{phi0}. Then 
\begin{enumerate}
\item  $\phi_0$ has the following integrability: for any $0\leq \beta\leq \alpha$, 
\[
\int_0^R w^{\alpha-\beta} r^4\phi_0^2dr +\int_0^R w^{1+\alpha -\beta} r^4 (\phi_0')^2dr \;<\infty. 
\]
Moreover, for any $z>1$, 
$$\int_0^R w^{z-2}{r^4\phi_0^2}dr <\infty.$$
\item $\phi_0$ has the following regularity: for $1\leq k\leq 2\alpha+1$, 
\[
\int_0^R w^{1+\alpha+k}r^4 (\partial_r^{k+1}\phi_0)^2 dr<\infty. 
\]
\end{enumerate} 
\end{lemma}

\begin{proof} We start with the first item. $\beta=0$ corresponds to the regularity of $H$ in which $\phi_0$ was constructed. From the regularity of $H$, and by using the Hardy inequality --  see Lemma \ref{hardy-G} as well as \eqref{Hardy-gw} -- we deduce that 
\[
\int_0^R w^{\alpha-1}r^4\phi_0^2 dr \precsim \int_0^R w^{\alpha+1}r^4\phi_0^2 dr + \int_0^R w^{\alpha+1}r^4(\phi_0')^2 dr <\infty. 
\]
In order to bootstrap this better integrability for $\phi_0$ into the one for $\phi_0'$, we will make use of the elliptic structure of the equation \eqref{phi0}.  Multiply \eqref{phi0} by $w^{-1}\phi_0$ and integrate it over $(0,R)$: 
\begin{equation}\label{astast}
\begin{split}
\underbrace{\int_0^R-\frac{1+\alpha}{\alpha}\frac{\phi_0}{w} [w^{1+\alpha} r^4\phi_0']'  dr }_{(\ast)}&+\int_0^R  \mu_0 w^{\alpha-1} r^4\phi_0^2 dr \\
& = \underbrace{\int_0^R  - (4-3\frac{1+\alpha}{\alpha} )\frac{{[w^{1+\alpha}]'}}{w}\, r^3 \phi_0^2 dr}_{(\ast\ast)}
\end{split}
\end{equation}
Note that by Lemma \ref{lem2.1}
\begin{equation}\label{ast2}
(\ast\ast)\precsim  \int_0^R w^{\alpha-1}r^4\phi_0^2 dr<\infty. 
\end{equation}
For the $(\ast)$, by integration by parts, we obtain 
\[
\begin{split}
\frac{\alpha}{1+\alpha}(\ast)& =   \int_0^R w^{\alpha} r^4 (\phi_0')^2 dr  -\int_0^R w' w^{\alpha-1} r^4 \phi_0\phi_0' dr\\
&=  \int_0^R w^{\alpha} r^4 (\phi_0')^2 dr  +\frac{\alpha-1}{2}\int_0^R  
(w')^2 w^{\alpha-2} r^4 \phi_0^2 dr \\
&\quad+ \underbrace{\frac12\int w'' w^{\alpha-1} r^4 (\phi_0)^2 dr + 2\int_0^R w'w^{\alpha-1} r^3 \phi_0^2 dr}_{(\ast)_1}
\end{split}
\]
From \eqref{L-E}, we see that $$|(\ast)_1|\precsim  \int_0^R w^{2\alpha-1}r^4\phi_0^2 dr+ \int_0^R w^{\alpha-1}r^4\phi_0^2 dr<\infty$$ and hence, together with \eqref{ast2} from \eqref{astast} we deduce that 
\[
\int_0^R w^{\alpha} r^4 (\phi_0')^2 dr  +\frac{\alpha-1}{2}\int_0^R 
(w')^2 w^{\alpha-2} r^4 \phi_0^2 dr <\infty. 
\]
This establishes the case of $\beta=1$. The general case can be done by a similar bootstrapping argument: assume that $\int_0^Rw^{\alpha-\beta}r^4\phi_0^2 dr<\infty$ for $0< \beta\leq \alpha$. By multiplying \eqref{phi0} by 
$w^{-\beta} \phi_0$ and integrating to derive that $$\int_0^R w^{\alpha+1-\beta} r^4(\phi_0')^2dr +\frac{\alpha-\beta}{2}\int_0^R 
(w')^2 w^{\alpha-\beta-1} r^4 \phi_0^2 dr$$ is bounded.  We remark that the condition on $\beta$ is necessary to keep the second integral stay non-negative; otherwise, we won't be to close the argument. Now by the Hardy inequality again in Lemma \ref{hardy-G}, we conclude that $\int_0^R w^{z-2}{r^4\phi_0^2}dr$ is bounded as long as $z>1$. 

We move onto the item 2 regarding the higher regularity of $\phi_0$. The key to achieve this higher regularity for $\phi_0$ is to capture the right behavior of $L$ with respect to the spatial derivatives. We will follow the idea in \cite{JM10}. We first write \eqref{phi0} as follows: 
\begin{equation}\label{3.15}
-\left[w\phi_0'' +(1+\alpha) w'\phi_0'+\frac{4w}{r}\phi_0'\right] = - \left[  \alpha(4-3\frac{\alpha}{1+\alpha}) \frac{w'}{r}\phi_0 +\frac{\alpha\mu_0}{1+\alpha} \phi_0 \right]
\end{equation}
Take the derivative of \eqref{3.15} and write it as 
\begin{equation}\label{3.16}
\begin{split}
-\left[w\phi_0''' +(2+\alpha) w'\phi_0''+\frac{4w}{r}\phi_0''\right] &=(1+\alpha) w''\phi_0'+(\frac{4w}{r})'\phi_0' \\
\quad&- \left[  \alpha(4-3\frac{\alpha}{1+\alpha}) \frac{w'}{r}\phi_0 +\frac{\alpha\mu_0}{1+\alpha} \phi_0 \right]'
\end{split}
\end{equation}
so that the left-hand-side of \eqref{3.16} keeps the self-adjoint structure with different weights: 
\[
-\left[w\phi_0''' +(2+\alpha) w'\phi_0''+\frac{4w}{r}\phi_0''\right]= -\frac{1}{w^{1+\alpha} r^4}\left[ w^{2+\alpha}r^4 \phi_0''\right]'
\]
Hence, by multiplying \eqref{3.16} by $w^{1+\alpha} r^4\phi_0'$ and integrating, we get 
\[
\begin{split}
\int_0^R w^{2+\alpha} r^4(\phi_0'')^2dr & = \int_0^R [(1+\alpha)w'' +(\frac{4w}{r})'] w^{1+\alpha} r^4  ( \phi_0')^2 dr  \\
&- \int_0^R   \left[  \alpha(4-3\frac{\alpha}{1+\alpha}) \frac{w'}{r}\phi_0 +\frac{\alpha\mu_0}{1+\alpha} \phi_0 \right]' w^{1+\alpha} r^4   \phi_0'dr
\end{split}
\]
It is easy to see that the right-hand-side is bounded because of the property of $w$ and since $\phi_0\in H$. The higher-order estimates can be inductively done in the same manner. By taking more derivatives of \eqref{3.16}, one can derive that 
\begin{equation}\label{3.17}
\begin{split}
-\left[w\partial_r^{k+2}\phi_0 +(k+1+\alpha) w'\partial_r^{k+1}\phi_0+\frac{4w}{r}\partial_r^{k+1}\phi_0\right] = \mathcal{T}
\end{split}
\end{equation}
where $\mathcal{T}$ consists of lower-order terms such as $(\partial_r^{k+1}w)\phi_0',(\partial_r^{k}w)\phi_0'',...,w''\partial_r^{k}\phi_0$ and so on. The left-hand-side of \eqref{3.17} is written as 
\[
-\left[w\partial_r^{k+2}\phi_0 +(k+1+\alpha) w'\partial_r^{k+1}\phi_0+\frac{4w}{r}\partial_r^{k+1}\phi_0\right] = - 
\frac{1}{w^{k+\alpha} r^4}\left[ w^{k+1+\alpha}r^4 \partial_r^{k+1}\phi_0\right]'
\]
and thus by multiplying by $w^{k+\alpha} r^4\partial_r^{k}\phi_0$ and integrating, we get 
\[
\int_0^R w^{1+\alpha+k} r^4|\partial_r^{k+1}\phi_0|^2 dr =\int_0^R \mathcal{T}\cdot w^{k+\alpha} r^4\partial_r^{k}\phi_0 dr 
\]
What remains to show is that the right-hand-side is bounded. Suppose that $\mathcal{T}$ has the form of 
$(\partial_r^{j+1}w) (\partial_r^{k+1-j}\phi_0)$ for $1\leq j\leq k $; other terms are even lower than these terms and they can be treated in the same way. Then 
\[
\int_0^R \mathcal{T}\cdot w^{k+\alpha} r^4\partial_r^{k}\phi_0 dr \Rightarrow \int_0^R(w^{\frac{j-1}{2}}\partial_r^{j+1}w) (w^{\frac{k+1-j+\alpha}{2}}r^2\partial_r^{k+1-j}\phi_0) (w^{\frac{k+\alpha}{2}} r^2\partial_r^{k}\phi_0) dr 
\]
and thus by using Lemma \ref{lem2.1} and by using Cauchy-Swartz inequality, we deduce that it is bounded. Hence, we conclude that $\int_0^R w^{1+\alpha+k} r^4|\partial_r^{k+1}\phi_0|^2 dr<\infty$ as long as $1\leq k\leq 2\alpha+1$. This completes the proof the lemma.  
\end{proof}

\begin{remark}\label{rem3.7}
Lemma \ref{lem3.3} and \ref{lem3.4} show that the growing mode $\phi_0$ does not display  an erratic nor singular behavior near the origin and the boundary. Indeed, this growing mode profile belongs to the function spaces of our interest, namely has a finite total initial energy for $3<\alpha<5$; see \eqref{Ejk} and \eqref{total}. And hence, we can use it as an initial perturbation for the nonlinear problem. 
\end{remark}

\subsection{The linear growth rate}
 
We will now show that $\sqrt{\mu_0}$ is the dominating growth rate for the linearized Euler-Poisson equation.     
We begin by writing the linearized equation \eqref{zetaL}  for $\Psi\equiv\zeta$
\begin{equation}\label{L1}
w^\alpha r^4\Psi_{tt} = L\Psi  
\end{equation}
where $L$ is the self-adjoint linear operator given in \eqref{L}. Here, instead of $\zeta$, we use $\Psi$ to distinguish the linear analysis from the forthcoming nonlinear analysis and we note that $\Psi=\Psi(t,r)$ while the $\phi$  in \eqref{lambda2} and \eqref{L} depends only on $r$. 

We now introduce the following weighted norms: 
\begin{equation}
\begin{split}\label{XY}
\| f\|^2_{X^\alpha}&\equiv \int_0^R w^\alpha r^4 f^2 dr, \\
\| f\|^2_{Y^\alpha}&\equiv \frac{1+\alpha}{\alpha}\int_0^R w^{\alpha+1} r^4 f_r^2 dr. 
\end{split}
\end{equation}

The first result is on the estimate of $\Psi$ satisfying \eqref{L1} in $\|\cdot\|_{X^\alpha}$. 

\begin{lemma}\label{lem4.1}
For every solution $\Psi$ to \eqref{L1} there exist
$C_{\mu_0}, C_{\mu_0,i}>0$ such that
\begin{equation}
\|\Psi(t)\|_{X^\alpha}\text{, } \|\Psi_t(t)\|_{X^\alpha} \leq C_{\mu_0}
e^{\sqrt{\mu_0}
t}(\|\Psi_t(0)\|_{X^\alpha}+\|\Psi(0)\|_{X^\alpha}+\|\Psi(0)\|_{Y^\alpha})\label{(1)}
\end{equation}
and for any $i\geq 1$,
\begin{equation}
\begin{split}
\|\partial_t^{i+1}\Psi(t)\|_{X^\alpha} \leq C_{\mu_0,i}&
e^{\sqrt{\mu_0}
t}(\|\Psi_t(0)\|_{X^\alpha}+\|\Psi(0)\|_{X^\alpha}+\|\Psi(0)\|_{Y^\alpha})\\
+
C&_{\mu_0,i}\sum_{j=1}^{i}(\|\partial_t^{j+1}\Psi(0)\|_{X^\alpha}+
\|\partial_t^j\Psi(0)\|_{Y^\alpha}).\label{(2)}
\end{split}
\end{equation}
\end{lemma}

\begin{proof}

Take the inner product of \eqref{L1} with $\Psi_t$. Then we
get
$$(w^\alpha r^4 \Psi_{tt}, \Psi_t)=(L\Psi, \Psi_t) \;\Longleftrightarrow\;\frac{d}{dt}(w^\alpha r^4 \Psi_{t}, \Psi_t)=\frac{d}{dt}(L \Psi,
\Psi).$$ The above equivalence comes from the self-adjointness of
$L$. Next integrate the above in time $t$ to get
\begin{equation}
(w^\alpha r^4
\Psi_t(t),\Psi_t(t))=(L\Psi(t),\Psi(t))+(w^\alpha r^4\Psi_t(0),\Psi_t(0))
-(L\Psi(0), \Psi(0)).\label{4.4}
\end{equation}
Since  $(L \Psi(t),\Psi(t)) \leq \mu_0 (w^\alpha r^4 \Psi(t), \Psi(t))$ for
all $t$ because of the definition of $\mu_0$ and since $-(L\Psi(0), \Psi(0))\leq \|\Psi(0)\|^2_{Y^\alpha}$, from
(\ref{4.4}) we see that 
\begin{equation}
\|\Psi_t(t)\|_{X^\alpha}^2 \leq \mu_0 \|\Psi(t)\|_{X^\alpha}^2 +
\|\Psi_t(0)\|_{X^\alpha}^2 + \|\Psi (0)\|^2_{Y^\alpha}\label{4.5}. 
\end{equation}
Now since $$\|\Psi(t)\|_{X^\alpha} \leq \int_0^t \|\Psi_t(\tau)\|_{X^\alpha} d\tau +
\|\Psi(0)\|_{X^\alpha}$$
plugging this into (\ref{4.5}), we get
\begin{equation*}
\|\Psi_t(t)\|_{X^\alpha} \leq \sqrt{\mu_0} \int_0^t
\|\Psi_t(\tau)\|_{X^\alpha}d\tau +
C(\|\Psi_t(0)\|_{X^\alpha}+\|\Psi(0)\|_{X^\alpha}+\| \Psi (0)\|_{Y^\alpha} ).
\end{equation*}
By the Gronwall inequality, we obtain
\begin{equation*}
\begin{split}
\|\Psi_t(t)\|_{X^\alpha} &\leq C e^{\sqrt{\mu_0}
t}(\|\Psi_t(0)\|_{X^\alpha}+\|\Psi(0)\|_{X^\alpha}+\| \Psi (0)\|_{Y^\alpha})
\text{ and in success }\\
\
\|\Psi(t)\|_{X^\alpha}\; &\leq C e^{\sqrt{\mu_0}  t}(
\|\Psi_t(0)\|_{X^\alpha}+\|\Psi(0)\|_{X^\alpha}+\| \Psi (0)\|_{Y^\alpha}).
\end{split}
\end{equation*}
Note that $C$ only depends on $\mu_0$. For higher derivatives,
take $\partial_t^{\alpha}$ of (\ref{L1}):
$w^\alpha r^4\partial_t^{i}\Psi_{tt} = L\partial_t^{i} \Psi$.
Take the inner product of this with $\partial_t
\partial_t^{i} \Psi$ to get
\begin{equation*}
\begin{split}
\|\partial&_t^{i+1}\Psi(t)\|_{X^\alpha}^2 \\&=(w^\alpha r^4
\partial_t^{i+1}\Psi(t),\partial_t^{i+1}\Psi(t))\\
&=(L\partial_t^{i}\Psi(t),\partial_t^{i}\Psi(t))+
(w^\alpha r^4\partial_t^{i+1}\Psi_t(0),\partial_t^{i+1}\Psi_t(0))
-(L\partial_t^{i}\Psi(0), \partial_t^{i}\Psi(0))\\
&\leq \mu_0\|\partial_t^{i}\Psi(t)\|_{X^\alpha}^2 +
\|\partial_t^{i+1}\Psi(0)\|_{X^\alpha}^2 +
\|\partial_t^i\Psi(0)\|^2_{Y^\alpha}.
\end{split}
\end{equation*}
Hence (\ref{(2)})  follows.
\end{proof}

The estimates of $\Psi$ in $Y^\alpha$ space can be derived by using the self-adjointness of $L$ again and by Lemma \ref{lem4.1}. 

\begin{lemma}\label{lem3.7} For every solution $\Psi$ to \eqref{L1} there exists $C_{\mu_0}$ and $C_{\mu_0,i}>0$ such that 
\[
\|\Psi (t)\|^2_{Y^\alpha}
\leq C_{\mu_0}e^{2\sqrt{\mu_0} t}(\|\Psi_t(0)\|^2_{X^\alpha}+\|\Psi(0)\|^2_{X^\alpha}+\|\Psi(0)\|^2_{Y^\alpha})
\]
and for $i\geq 1$, 
\[
\begin{split}
\|\partial_t^{i}\Psi(t)\|^2_{Y^\alpha} \leq C_{\mu_0,i}&
e^{2\sqrt{\mu_0}
t}(\|\Psi_t(0)\|^2_{X^\alpha}+\|\Psi(0)\|^2_{X^\alpha}+\|\Psi(0)\|^2_{Y^\alpha})\\
+
C&_{\mu_0,i}\sum_{j=1}^{i}(\|\partial_t^{j+1}\Psi(0)\|^2_{X^\alpha}+
\|\partial_t^j\Psi(0)\|^2_{Y^\alpha}).
\end{split}
\]
\end{lemma}

\begin{proof} Here we give a proof for $i=0$. Other cases of $i\geq 1$ can be treated in the same way. By using \eqref{L},  we rewrite the equation \eqref{4.4} as 
\[
\|\Psi(t)\|_{Y^\alpha}^2=-\|\Psi_t(t)\|_{X^\alpha}^2-(4-3\frac{1+\alpha}{\alpha})\int_0^R[w^{1+\alpha}]' r^3\Psi^2 dr +\|\Psi_t(0)\|_{X^\alpha}^2- (L\Psi(0), \Psi(0)). 
\]
Notice that since $-\frac{[w^{1+\alpha}]'}{r}\precsim w^\alpha$ from \eqref{EQ}, the second term in the right-hand-side is bounded by $ C\|\Psi(t)\|^2_{X^\alpha}$.  Moreover, by recalling that $-(L\Psi(0), \Psi(0))\leq \|\Psi(0)\|^2_{Y^\alpha}$,  we  deduce that 
\[
\|\Psi(t)\|_{Y^\alpha}^2\leq C \|\Psi(t)\|^2_{X^\alpha}+ \|\Psi_t(0)\|_{X^\alpha}^2 +\|\Psi(0)\|_{Y^\alpha}^2. 
\]
And therefore, from Lemma \ref{lem4.1}, the conclusion follows.
\end{proof}

In the next two sections, we develop the nonlinear theory for solutions to the Euler-Poisson equation  in the spherically symmetric motion \eqref{zeta}.


\section{The instant energy and total energy}\label{4}

 In this section, we will introduce the various energies  and establish the equivalence of the temporal instant energy and the total energy for the solutions to  the Euler-Poisson equation. 

 We first rewrite the equation \eqref{zeta} as 
\begin{equation}\label{zetaE}
\begin{split}
\frac{w^\alpha r^4\zeta_{tt}}{(1+\zeta)^2} &-\frac{1+\alpha}{\alpha} (w^{1+\alpha} r^4 \zeta_r)_r+(3\frac{1+\alpha}{\alpha} -4) w^\alpha r^4\Phi(r) \zeta \\
&+r^3(w^{1+\alpha} h(\zeta,\zeta_r r))_r-w^\alpha r^4\Phi(r) f(\zeta)=0
\end{split}
\end{equation}
where $\Phi(r)$ is the prescribed function defined by 
\begin{equation}\label{Phi}
\begin{split}
\Phi(r)&\equiv\frac{1}{r^3}\int_0^r\frac{4\pi}{K^\alpha}w^\alpha s^2 ds=-\frac{(w^{1+\alpha})_r}{rw^\alpha} = -(1+\alpha) \frac{w_r}{r}
\end{split}
\end{equation}
and $f(\zeta)$ is given by 
\[
f(\zeta)\equiv \frac{(1+\zeta)^4-1-4\zeta(1+\zeta)^4}{(1+\zeta)^4}
\]
and thus $f(\zeta)=O(|\zeta|^2)$ for small $\zeta$, and finally $h$ is given in  \eqref{exp}.

We introduce the following instant energies and the total energy for the solutions to the nonlinear Euler-Poisson  equation \eqref{zetaE}. We denote the zeroth-order kinetic energy plus the lower-order potential energy for $\zeta$ by $\mathcal{E}^0$: 
\begin{equation}\label{E0}
\mathcal{E}^0\equiv\int_0^R w^\alpha r^4 \left| \zeta_t\right|^2 + \frac{1+\alpha}{\alpha}  w^{1+\alpha} r^4\left | \zeta_r\right|^2 dr +\int_0^R w^\alpha r^4\left| \zeta\right|^2 dr 
\end{equation}
The higher-order (temporal) instant energy is denoted by $\mathcal{E}^j$ for $j\geq 1$: 
\begin{equation*}
\mathcal{E}^j\equiv \int_0^Rw^\alpha r^4 \big|\partial_t^{j} \zeta_t\big|^2 + \frac{1+\alpha}{\alpha}  w^{1+\alpha} r^4 \big|\partial_t^j\zeta_r\big|^2 dr 
\end{equation*}
and the total instant energy by $\mathcal{E}(t)$: 
\begin{equation}\label{TE}
\mathcal{E}(t) \equiv \mathcal{E}^0(t)+\sum_{j=1}^{\lceil\alpha\rceil+3} \mathcal{E}^j(t)
\end{equation}
where  $\lceil \alpha \rceil$ is a ceiling function, namely
$\lceil \alpha \rceil = \min \{n\in \mathbb{Z}: \alpha\leq n\}.$  
Notice that the instant energies can be expressed in terms of $X^\alpha$ and $Y^\alpha$ that were introduced earlier in \eqref{XY} as follows: 
\[
\begin{split}
\mathcal{E}^0& =\| \zeta_t \|^2_{X^\alpha} +\| \zeta \|^2_{Y^\alpha} +\| \zeta \|^2_{X^\alpha}\\
\mathcal{E}^j&= \| \partial_t^j\zeta_t \|^2_{X^\alpha} +\| \partial_t^j\zeta \|^2_{Y^\alpha}
\end{split}
\]
For the mixed derivatives, we introduce the following notation: 
\begin{equation}\label{Ejk}
\mathcal{E}^{j,k}\equiv \int_0^Rw^{\alpha+k} r^4\big |\partial_t^{j-k} \partial_r^k \zeta_t\big|^2 + \frac{1+\alpha}{\alpha}  w^{1+\alpha+k} r^4 \big|\partial_t^{j-k}\partial_r^k\zeta_r\big|^2 dr. 
\end{equation} 
We remark that this change of weights for the spatial derivatives in the energy norms is not a coincidence but rather  a nature of the physical vacuum state of polytropic gases; this nature was discovered and analyzed in \cite{JM, JM10} for general $\alpha$.  The total energy including both temporal and spatial derivatives is defined by 
\begin{equation}\label{total}
\widetilde{\mathcal{E}}(t) \equiv \mathcal{E}^0(t) +\sum_{j=1}^{\lceil\alpha\rceil+3} \sum_{k=0}^j\mathcal{E}^{j,k}. 
\end{equation}
It is clear that 
\[
\mathcal{E}(t)\leq \widetilde{\mathcal{E}}(t). 
\]
We want to show that the converse holds for the solutions to \eqref{zetaE} in such a way: $\widetilde{\mathcal{E}}$ is bounded by $\mathcal{E}$ under the following  smallness assumption: 
\begin{equation}\label{small}
\big|\zeta\big|+\big|\zeta_r\big|+ \sum_{q=1}^{[\frac{\lceil\alpha \rceil+3}{2}]+1} \big|w^{\frac{q-1}{2}} \partial_t^q\zeta \big|
+ \sum_{q=1}^{[\frac{\lceil\alpha \rceil+3}{2}]}\big|w^{\frac{q-1}{2}}\partial_t^q\zeta_{r}\big|    \leq \theta_1
\end{equation}
where $\theta_1$ is a sufficiently small, fixed constant.  The validity of this assumption within the total energy $\widetilde{\mathcal{E}}$ will be justified in Lemma \ref{lem4.7}.

Our main goal in this section is to prove the following proposition, which concerns the equivalence of $\mathcal{E}(t)$ and $\widetilde{\mathcal{E}}(t)$ under the smallness assumption \eqref{small}. 

\begin{proposition}\label{equiv} Let $(\zeta_r,\zeta)$ be a solution to \eqref{zetaE} and moreover assume that \eqref{small} holds for  $0\leq t\leq T$. Then there exists a constant $C>0$ such that for $0\leq t\leq T$
\begin{equation*}
\widetilde{\mathcal{E}}(t)\leq C \mathcal{E}(t). 
\end{equation*}
\end{proposition}

In order to prove Proposition \ref{equiv}, we first recall standard Hardy inequalities, and embedding inequalities, and derive more of Hardy type inequalities which can be adapted to our energy spaces induced by $\mathcal{E}$.  The proof of Proposition \ref{equiv} will be given in Section \ref{sec4.2}.

\subsection{Hardy inequalities}

The following is the most well-known Hardy inequality applied to the Sobolev spaces: 
\begin{equation}\label{har}
\int_0^\infty  {\left| \frac{g(x)}{x}\right|}^2dx\;\precsim \; \int_0^\infty |g'(x)|^2 dx. 
\end{equation} 
We will also make use of the more general version of the  Hardy 
inequality:

\begin{lemma}[Hardy inequality] \label{hardy-G} 
Let  $k $ be a given real number. And let $g$ be a function satisfying $\int_0^1  s^k (g^2 + g'^2) ds < \infty. $
\begin{enumerate}
\item Then if $k > 1$, then we have $$ \int_0^1  s^{k-2}  g^2 ds  \leq C \int_0^1  s^k (g^2 + |g'|^2) ds.   $$ 

\item And  if $k < 1$, then $g$ has a trace at $x=0$  and moreover, 
$$ \int_0^1  s^{k-2}  (g - g(0))^2 ds  \leq C \int_0^1  s^k   |g'|^2  ds  .$$ 
\end{enumerate}
 \end{lemma}

For the proof of Lemma \ref{hardy-G}, we refer to  \cite{KMP07}. 

As an application of the standard Hardy inequality \eqref{har}, we first obtain the following Hardy inequality  localized  near the origin. 

\begin{lemma}\label{hard} Let $u\in X^a\cap Y^a$ be given where $X^a$ and $Y^a$ are defined in \eqref{XY} and $a$ is any real number. Let $c$ be a positive fixed number such that $0<c<2c<R$. Then there exists a constant $C>0$ independent of $u$ such that 
\begin{equation}\label{hardy0}
\int_0^c r^2 |u|^2 dr \leq C\left (\int_0^{2c} r^4 |u_r|^2  dr +\int_c^{2c} r^4 | u|^2 dr \right).
\end{equation}
\end{lemma}

Note that the right-hand-side of \eqref{hardy0} is bounded by $\|u\|_{X^a}^2$ and  $\|u\|_{Y^a}^2$ since the weight $w$ is bounded from above and below from zero on $(0,2c)$. 

\begin{proof} Suppose $u$ is a smooth function. Choose a smooth cutoff function $\chi\in C^\infty(0,\infty)$ such that $0\leq \chi\leq 1$, $\chi=1$ on $(0,c)$, and $\chi=0$ for $r\geq 2c$. Then 
\[
\begin{split}
\int_0^{2c} r^2 |\chi u|^2 dr
& =\int_0^{\infty}  |\frac{ r^2\chi u}{r}|^2 dr \\
&\leq C \int_0^\infty | ( r^2\chi u)_r|^2dr \;\text{ by Hardy inequality \eqref{har}}\\
&= C\int_0^{2c} (r^2(\chi u)_r+2r \chi u )^2 dr\\
&=C (\int_0^{2c} r^4|(\chi u)_r|^2 dr +4\int_0^{2c} r^3\chi u(\chi u)_r dr +4\int_0^{2c} r^2(\chi u)^2 dr )\\
&= C(\int_0^{2c} r^4|(\chi u)_r|^2 dr -2\int_0^{2c} r^2(\chi u)^2 dr)\;\text{ by integration by parts}
\end{split}
\]
which yields that 
\[
\int_0^{2c} r^2 |\chi u|^2 dr\leq C \int_0^{2c} r^4|(\chi u)_r|^2 dr. 
\]
Hence, we obtain 
\[
\begin{split}
  \int_0^c r^2 | u|^2 dr \leq \int_0^{2c} r^2 |\chi u|^2 dr&\leq C \int_0^{2c} r^4|(\chi u)_r|^2 dr\\
  &\leq C\Big( \int_0^{2c} r^4 \chi^2u_r^2 dr+ \int_0^{2c} r^4 \chi_r^2 u^2 dr\Big)\\
&\leq C\Big(\int_0^{2c} r^4 |u_r|^2  dr +\int_c^{2c} r^4 | u|^2 dr\Big). 
\end{split}
\]
The inequality \eqref{hardy0} is easily extended to $u\in X^a\cap Y^a$ by the density argument. 
\end{proof}

The term as in the left-hand-side of \eqref{hardy0}, which has the stronger weight near the origin,  will appear in the subsequent nonlinear estimates  due to the choice of the spherical coordinates in our analysis. 
Lemma \ref{hard} asserts that such a coordinate singularity  can be resolved in our energy spaces by means of Hardy inequality.  We can apply the same argument for $u=\frac{v}{r}$:
\[
\int_0^R |\chi v|^2 dr \leq C \left(\int_0^R r^2 |\chi v_r|^2 dr+ \int_0^R r^2 |\chi v|^2 dr\right)
\]
and then apply \eqref{hardy0} to both terms to obtain the following bootstrapping estimates near the origin: 
\begin{equation}\label{bts}
\int_0^R |\chi v|^2 dr \leq C\left(\int_0^R r^4 |\chi v_{rr}|^2 dr+\int_0^R r^4 |\chi v_r|^2 dr+ \int_0^R  r^4 |\chi v|^2 dr\right). 
\end{equation}

Similarly, by using the general Hardy inequality in Lemma \ref{hardy-G}, we can derive the following Hardy inequality localized near the boundary: 
\begin{equation}\label{Hardy-gw}
\int_0^R w^{a-2}| \psi v|^2 dr \leq C\left( \int_0^R w^a |\psi v_r|^2 dr +\int_0^R w^a |\psi v|^2 dr\right)
\end{equation}
where $\psi$ is the smooth cutoff function satisfying $0\leq \psi\leq 1$, $\psi=0$ for $r\leq R-2c$ and $\psi=1$ for $r\geq R-c$.

Next we recall the Hardy type embedding for the weighted Sobolev spaces which will be importantly used for the nonlinear energy estimates. Let $\Omega$ be a smooth domain in $\mathbb{R}^n$ and let $d=d(x)=\text{dist}(x,\partial\Omega)$ be a distance function to the boundary. For any given $a>0$ and given nonnegative integer $b$, we define the weighted Sobolev spaces $H^{a,b}(\Omega)$ by 
\[
H^{a,b}(\Omega)\equiv \{d^{\frac{a}{2}}u\in L^2(\Omega) : \int_{\Omega}d^{a}|\nabla^k u|^2 dx<\infty\,,\,0\leq k\leq b\}
\]
with the norm 
\[
\|u\|^2_{H^{a, b}}\equiv \sum_{k=0}^b  \int_{\Omega}d^{a}|\nabla^k u|^2 dx\,.
\]
We denote the standard Sobolev spaces by $H^s$.  Then for $b\geq a/2$, the weighted spaces $H^{a,b}$ satisfy the following  Hardy type embedding \cite{KMP07}: 
\[
H^{a,b}(\Omega)\hookrightarrow H^{b-\frac{a}{2}}
\]
with the estimate 
\begin{equation}\label{hardy}
\|u\|_{H^{b-a/2}}\;\precsim \;\|u\|_{H^{a,b}}\,.
\end{equation}
For instance, the inequality \eqref{bts} can be obtained from the localization (by the cutoff technique as in the proof of Lemma \ref{hard}) of \eqref{hardy} near the origin by taking $\Omega=(0,R)$, $a=4$ and $b=2$.  

We now turn our attention to the degeneracy near the boundary $r=R$. We would like to adapt  the inequality  \eqref{hardy} to our energy spaces \eqref{total}. Since there is loss of the weight $w$ for each spatial derivative, the space for $(\zeta_t, \zeta)$ generated by $\sum_{k=0}^n\mathcal{E}^{k,k}$ near the boundary  is equivalent to the localization of $ \cap_{k=0}^{n}H^{\alpha+k,k}\times H^{\alpha+k+1,k+1}$ near the boundary  with the above notation. 
Notice that $ \cap_{k=0}^{n} H^{\alpha+k+1,k+1}\subset H^{\alpha+n+1,n+1}$. Therefore, it is desirable to  derive some Hardy inequalities by using these spaces: $H^{\beta+l,l}$ for a given positive number $\beta>0$ and positive integers $l>0$. First, by applying the Hardy embedding inequality \eqref{hardy} for $a=\beta+l$ and $b=l$, we obtain 
\begin{equation}\label{hardy.}
\|\psi u\|_{H^{\frac{l-\beta}{2}}}\precsim \|\psi u\|_{H^{\beta+l,l}}
\end{equation}
where $\psi$ is the smooth cutoff function satisfying $0\leq \psi\leq 1$, $\psi=0$ for $r\leq R-2c$ and $\psi=1$ for $r\geq R-c$. Hence, by letting $l=\lceil \beta\rceil$, we obtain the following: 
\[
\int_0^R |\psi u|^2 dr \precsim \sum_{k=0}^{\lceil \beta\rceil }\int_0^Rw^{\beta +\lceil \beta\rceil} |\partial_r^k (\psi u)|^2 dr \leq \sum_{k=0}^{\lceil \beta\rceil }\int_0^Rw^{\beta +k} |\partial_r^k (\psi u)|^2 dr
\]
If $\lceil \beta \rceil\geq2$, we can combine with the Hardy inequality near the origin \eqref{bts} to derive the following: 
\begin{equation*}
\int_0^R u^2 dr \leq C\sum_{k=0}^{\lceil \beta\rceil }\int_0^Rw^{\beta +k} r^4 |\partial_r^k  u|^2 dr. 
\end{equation*}
Moreover, if we choose $l=\lceil \beta\rceil+2m$ for $m\geq0$ in \eqref{hardy.} and if we combine with  the Hardy inequality near the origin \eqref{bts} when $\lceil \beta \rceil +2m\geq m+2$, 
we further obtain 
\begin{equation}\label{45}
\|u\|^2_{H^m} \leq C\sum_{k=0}^{\lceil \beta\rceil +2m}\int_0^Rw^{\beta +k} r^4 |\partial_r^k  u|^2 dr. 
\end{equation} 
We have established the following.  

\begin{lemma}\label{lem4.3} Let $\partial_r^k u\in X^{\beta +k}$ for each $0\leq k\leq \lceil\beta\rceil +2m$ where $\lceil \beta \rceil +m\geq2$. Then $u\in H^m(0,R)$ and moreover, the inequality \eqref{45} holds. Similar results are valid for $Y^{\beta}$: if  $\partial_r^k u\in Y^{\beta +k}$ for each $0\leq k\leq \lceil \beta\rceil +1+2m$,  then $u\in H^m(0,R)$. 
\end{lemma}

As a result of Lemma \ref{lem4.3}, we get the following $L^\infty$ embedding of our energy spaces: 

\begin{lemma}\label{lem4.4} Let $\partial_r^k u\in X^{\beta +k}$ for each $0\leq k\leq \lceil\beta\rceil +2$ where $\lceil \beta \rceil\geq 1$. Then $u\in L^\infty(0,R)$ with the following estimate:  
\begin{equation}\label{590}
\|u\|_{\infty}^2\leq C \sum_{k=0}^{\lceil \beta\rceil +2}\int_0^Rw^{\beta +k} r^4 |\partial_r^k  u|^2 dr. 
\end{equation}
Similarly, if  $\partial_r^k u\in Y^{\beta +k}$ for each $0\leq k\leq \lceil \beta\rceil +3$, then $u\in L^\infty(0,R)$. 
\end{lemma}

\begin{proof} Since $\|u\|_{L^\infty(0,R)}\leq C \|u\|_{H^1(0,R)}$ from the Sobolev embedding inequality, by combining with \eqref{45}, we get the desired result. 
\end{proof}

 Lemma \ref{lem4.3} and \ref{lem4.4} can be directly used to obtain  the embedding estimates for $w^p u$ for $p\geq 0$. For instance, if we apply Lemma \ref{lem4.4} for $wu$ in place of $u$ and use \eqref{Hardy-gw} again, we get the corresponding embedding estimate for $wu$. Here we record the results.

\begin{lemma}\label{lem4.5} Let $p\geq 0$ be given.  The followings hold. 
\begin{enumerate}
\item 
Let $\partial_r^k u\in X^{\beta+p +k}$ for each $0\leq k\leq \lceil\beta\rceil +2$ where $\lceil \beta \rceil\geq1$. Then $w^{p/2} u\in L^\infty(0,R)$ with the following estimate:  
\begin{equation}\label{600}
\| w^{p/2} u\|_{\infty}^2\leq C \sum_{k=0}^{\lceil \beta\rceil +2}\int_0^Rw^{\beta+p +k} r^4 |\partial_r^k  u|^2 dr. 
\end{equation}
Similarly, if  $\partial_r^k u\in Y^{\beta+p +k}$ for each $0\leq k\leq \lceil \beta\rceil +3$, then $w^{p/2} u\in L^\infty(0,R)$. 

\item  Let $\partial_r^k u\in X^{\beta+p +k}$ for each $0\leq k\leq \lceil\beta\rceil $. Then $w^{p/2} r^2 u\in L^2(0,R)$, equivalently $u\in X^p$ with the following estimate:  
\begin{equation}\label{6000}
\int_0^R w^pr^4 u^2 dr \leq C \sum_{k=0}^{\lceil \beta\rceil }\int_0^Rw^{\beta+p +k} r^4 |\partial_r^k  u|^2 dr. 
\end{equation}
\end{enumerate}
\end{lemma}

\subsection{$L^\infty$ bounds}

A direct consequence of the above Hardy embedding inequalities is the validity of the smallness assumption \eqref{small} within our energy space induced by $\widetilde{\mathcal{E}}$. 

\begin{lemma}\label{lem4.7}
There exists $C>0$ so that
\[
\big|\zeta\big|+\big|\zeta_r\big|+ \sum_{q=1}^{[\frac{\lceil\alpha \rceil+3}{2}]+1} \big|w^{\frac{q-1}{2}} \partial_t^q\zeta \big|
+ \sum_{q=1}^{[\frac{\lceil\alpha \rceil+3}{2}]}\big|w^{\frac{q-1}{2}}\partial_t^q\zeta_{r}\big|  \leq C{ \widetilde{\mathcal{E}}\,}^{\frac12}. 
\]
\end{lemma}

\begin{proof} We will present the detail on the following three terms $$\zeta_r, \quad w^{ [\frac{\lceil\alpha \rceil+3}{2}]/2} \partial_t^{[\frac{\lceil\alpha \rceil+3}{2}]+1}\zeta, \quad  
rw^{ [\frac{\lceil\alpha \rceil+3}{2}-1]/2} \partial_t^{[\frac{\lceil\alpha \rceil+3}{2}]}\zeta_r.$$Other terms can be treated in the same way. To see the boundedness of $\zeta_r$, we apply Lemma \ref{lem4.4}: take $\beta=\alpha+1$ in \eqref{590} to deduce that 
\[
\|\zeta_r\|_{\infty}^2\leq C \sum_{k=0}^{\lceil \alpha\rceil +3}\int_0^Rw^{\alpha+1 +k} r^4 |\partial_r^k  \zeta_r|^2 dr\leq C \widetilde{\mathcal{E}} . 
\]
For $w^{ [\frac{\lceil\alpha \rceil+3}{2}]/2} \partial_t^{[\frac{\lceil\alpha \rceil+3}{2}]+1}\zeta$, we will apply Lemma \ref{lem4.5}. Take $\beta=\alpha - [\frac{\lceil\alpha \rceil+3}{2}]$ in \eqref{600}. Notice that when $3<\alpha\leq 4$, $\beta=\alpha-3$; when $4<\alpha<5$, $\beta=\alpha-4$ and therefore, $\lceil\beta\rceil=1$, which satisfies the condition on $\beta$ in Lemma \ref{lem4.5}. Then, we see that 
\[
\begin{split}
\| w^{ [\frac{\lceil\alpha \rceil+3}{2}]/2} \partial_t^{[\frac{\lceil\alpha \rceil+3}{2}]+1}\zeta \|^2_\infty
&\leq C \sum_{k=0}^{3}\int_0^Rw^{\alpha +k} r^4 |\partial_r^k  \partial_t^{[\frac{\lceil\alpha \rceil+3}{2}]+1}\zeta|^2 dr\\
& \leq  C \sum_{k=0}^{3} \mathcal{E}^{ [\frac{\lceil\alpha \rceil+3}{2}] +k, k} \,\text{ by  }\,\eqref{Ejk}. 
\end{split}
\]
Since $ [\frac{\lceil\alpha \rceil+3}{2}] +3 <  \lceil\alpha \rceil+3$ for $\alpha>3$, we conclude that $\| w^{ [\frac{\lceil\alpha \rceil+3}{2}]/2} \partial_t^{[\frac{\lceil\alpha \rceil+3}{2}]+1}\zeta \|^2_\infty\precsim \widetilde{\mathcal{E}}$. We will apply Lemma \ref{lem4.5} for the last term. Take $\beta=\alpha - [\frac{\lceil\alpha \rceil+3}{2}]$ so that $\lceil\beta\rceil=1$ in \eqref{600}. 
We first obtain 
\[
\| w^{ [\frac{\lceil\alpha \rceil+3}{2} -1]/2} \partial_t^{[\frac{\lceil\alpha \rceil+3}{2}]}\zeta_r \|^2_\infty
\leq C \sum_{k=0}^{3}\int_0^Rw^{\alpha -1+k} r^4 |\partial_r^{k}  \partial_t^{[\frac{\lceil\alpha \rceil+3}{2}]}\zeta_r|^2 dr.
\]
Now by applying the Hardy inequality \eqref{Hardy-gw} and by using the definition of the energy \eqref{Ejk}, we further deduce that 
\[
\begin{split}
\| w^{ [\frac{\lceil\alpha \rceil+3}{2} -1]/2} \partial_t^{[\frac{\lceil\alpha \rceil+3}{2}]}\zeta_r \|^2_\infty
&\leq C \sum_{k=0}^{4}\int_0^Rw^{\alpha +1+k} r^4 |\partial_r^{k}  \partial_t^{[\frac{\lceil\alpha \rceil+3}{2}]}\zeta_r|^2 dr  \\
& \leq  C \sum_{k=0}^{4} \mathcal{E}^{ [\frac{\lceil\alpha \rceil+3}{2}] +k, k} .  
\end{split}
\]
Since $ [\frac{\lceil\alpha \rceil+3}{2}] +4 \leq  \lceil\alpha \rceil+3$ for $\alpha>3$, we conclude that $\|w^{ [\frac{\lceil\alpha \rceil+3}{2} -1]/2} \partial_t^{[\frac{\lceil\alpha \rceil+3}{2}]}\zeta_r\|^2_\infty\precsim \widetilde{\mathcal{E}}$.
\end{proof}

We are now ready to prove Proposition \ref{equiv}: the equivalence of $\mathcal{E}(t)$ and  $\widetilde{\mathcal{E}}(t)$. 

\subsection{Proof of Proposition \ref{equiv}}\label{sec4.2}

Denote two energy terms in \eqref{Ejk} by $\mathcal{E}^{j,k}_t$ and $\mathcal{E}^{j,k}_r$ so that $$\mathcal{E}^{j,k}=\mathcal{E}^{j,k}_t +\tfrac{1+\alpha}{\alpha}\mathcal{E}^{j,k}_r$$ where 
\begin{equation}\label{Ejk0}
\mathcal{E}^{j,k}_t\equiv \int_0^Rw^{\alpha+k} r^4 \big|\partial_t^{j-k} \partial_r^k \zeta_t\big|^2dr\;\text{ and } \; \mathcal{E}^{j,k}_r \equiv \int_0^R  w^{1+\alpha+k} r^4\big |\partial_t^{j-k}\partial_r^k\zeta_r\big|^2 dr. 
\end{equation}

 Notice that by definition, 
\begin{equation}\label{Ej0}
\mathcal{E}^{j,0}=\mathcal{E}^j \;\text{ and }\; \mathcal{E}^{j,k}_t =\mathcal{E}^{j,k-1}_r \;\text{ for }1\leq k\leq j. 
\end{equation}

 We will prove the equivalence for the simplest case: $j=1$ and $k=1$ first and then move onto the general case $j\geq2$.

\begin{lemma}[$\mathcal{E}^{1,1}$]  There exists a constant  $C>0$ such that 
\[
\mathcal{E}^{1,1}\leq C (\mathcal{E}^0+\mathcal{E}^1) . 
\]
\end{lemma}

\begin{proof}
 In this case, because of \eqref{Ej0}, we only need to show that $ \int_0^R  w^{2+\alpha} r^4 |\zeta_{rr}|^2 dr$ is bounded by the temporal instant energy. We recall the equation  \eqref{zetaE} in the following form: 
\begin{equation}\label{zetaE2}
\begin{split}
 \gamma \big (w^{1+\alpha} r^4 \zeta_r\big)_r&= \frac{w^\alpha r^4\zeta_{tt}}{(1+\zeta)^2} +(3\gamma-4) w^\alpha r^4\Phi(r) \zeta- w^\alpha r^4\Phi(r) f(\zeta) \\&\quad+r^3(w^{1+\alpha} h(\zeta,\zeta_r r))_r 
\end{split}
\end{equation}
where $f(\zeta)=O(|\zeta|^2)$ and $h(\zeta,\zeta_r r)=O(|\zeta|^2)+O(|\zeta_rr|^2)$. For notational convenience, we have used $\gamma=(1+\alpha)/\alpha$. 
We will exploit the elliptic structure of the term in the left-hand-side of \eqref{zetaE2}. Square both sides of \eqref{zetaE2}, divide them by $w^\alpha r^4$ and integrate it over $(0,R)$ to get 
\begin{equation}
\begin{split}\label{49}
\underbrace{\int_0^R \frac{\gamma^2}{w^\alpha r^4} \left| \big (w^{1+\alpha} r^4 \zeta_r\big)_r \right |^2 dr}_{(\ast)}\precsim 
\underbrace{\int_0^R  \frac{w^\alpha r^4|\zeta_{tt}|^2}{(1+\zeta)^4} dr }_{(i)} +\underbrace{\int_0^Rw^\alpha r^4|\Phi(r) f(\zeta)|^2 dr}_{(ii)}  \\
 + \underbrace{\int_0^R w^\alpha r^4|\Phi(r) \zeta|^2 dr}_{(iii)} +\underbrace{\int_0^R\frac{1}{w^\alpha r^4} \left| r^3 \big (w^{1+\alpha} h(\zeta, \zeta_rr)\big)_r \right |^2 dr}_{(iv)}
\end{split}
\end{equation}
It is clear that 
\begin{equation}\label{i}
(i)\precsim \mathcal{E}^{1},\quad (ii)\precsim \theta_1^2\mathcal{E}^{0},\quad (iii)\precsim \mathcal{E}^{0}. 
\end{equation}
For $(iv)$, we first see that 
\[
r^3 \big (w^{1+\alpha} h(\zeta, \zeta_rr)\big)_r= w^{1+\alpha} r^4 \,\partial_2 h\, \zeta_{rr} +w^{1+\alpha} r^3(\partial_1h+\partial_2h) \zeta_r
+r^3(w^{1+\alpha})_r h 
\]
where  $\partial_1h$ and $\partial_2h$ mean the derivative of $h$ with respective to the first and second  argument respectively.  For instance, if $h(\zeta,\zeta_rr)=\zeta^2+(\zeta_r r)^2$, $\partial_1h=2\zeta$ and $\partial_2h=2\zeta_r r$. By using the notation $\Phi$ given \eqref{Phi}, we write $(w^{1+\alpha})_r=-rw^\alpha\Phi(r)$ and hence, we see that $(iv)$ is bounded by 
\[
{(iv)} \precsim \underbrace{\int_0^R w^{2+\alpha} r^4 |\partial_2h|^2|\zeta_{rr}|^2dr}_{\precsim \theta_1^2\mathcal{E}^{1,1}_r} + \underbrace{\int _0^Rw^{2+\alpha} r^2 |\partial_1h+\partial_2h|^2|\zeta_r|^2 dr}_{(iv)_1} 
+\underbrace{ \int_0^Rw^\alpha r^4|\Phi(r)|^2 h^2dr }_{(iv)_2}. 
\]
It is easy to see that the first term in the right-hand-side is bounded by $\theta_1^2\mathcal{E}^{1,1}_r$. For the second and third terms, we will employ Hardy inequalities; by using the Hardy inequality near the origin \eqref{hardy0} for $(iv)_1$ and the Hardy inequality near the boundary \eqref{Hardy-gw}  for $(iv)_2$ and noting that $ |\partial_1h+\partial_2h|^2\precsim\theta_1^2$ and $ |h|\precsim \theta_1^2$, we deduce  that 
\begin{equation}\label{ii}
(iv)\leq C \theta_1^2 \{\mathcal{E}^0+\mathcal{E}^{1,1}_r  \}. 
\end{equation}
We now turn our attention to the term $(\ast)$ in the left-hand-side of \eqref{49}. First notice that 
\[
\begin{split}
 (w^{1+\alpha} r^4 \zeta_r)_r&= w^{1+\alpha} r^4 \zeta_{rr} + 4w^{1+\alpha} r^3\zeta_r +(w^{1+\alpha})_r r^4\zeta_r\\
 &= w^{1+\alpha} r^4 \zeta_{rr} + 4w^{1+\alpha} r^3\zeta_r -w^\alpha r^5 \Phi(r)\zeta_r
\end{split}
\]
Thus the $(\ast)$ in \eqref{49} reads as 
\[
{(\ast)}= \gamma^2\int_0^R w^\alpha r^4 \left| w\zeta_{rr} +\frac{4w\zeta_r}{r} -r\Phi(r)\zeta_r\right|^2dr. 
\]
By expanding terms out, we see that 
\[
\begin{split}
\frac{(\ast)}{\gamma^2}&=\int_0^R w^{2+\alpha}r^4 |\zeta_{rr}|^2  dr+ 16\int_0^R w^{2+\alpha}r^2 |\zeta_r|^2 dr+\int _0^R w^{\alpha}r^6|\Phi(r)|^2|\zeta_r|^2  dr\\
&\;\;+\underbrace{8\int_0^R w^{2+ \alpha}r^3 \zeta_{rr}\zeta_r dr}_{(\ast)_1} \underbrace{ - 2\int_0^R w^{1+\alpha} r^5 \Phi(r) \zeta_{rr}\zeta_{r}dr}_{(\ast)_2}- 8 \int_0^R w^{1+\alpha} r^4\Phi(r)|\zeta_r|^2 dr
\end{split}
\]
For the $(\ast)_1$ and $(\ast)_2$, we integrate by parts to get 
\[
\begin{split}
(\ast)_1&=-4\int_0^R (w^{2+\alpha})_r r^3 |\zeta_r|^2 dr -12\int_0^R w^{2+\alpha} r^2|\zeta_r|^2 dr  \\
&=4\frac{2+\alpha}{1+\alpha}\int _0^Rw^{1+\alpha} r^4\Phi(r)|\zeta_r|^2 dr  -12\int_0^R w^{2+\alpha} r^2|\zeta_r|^2 dr \\
(\ast)_2&=-\int_0^R w^\alpha r^6|\Phi(r)|^2 |\zeta_r|^2 +5\int_0^R w^{1+\alpha}r^4\Phi(r) |\zeta_r|^2  dr +\int_0^R w^{1+\alpha}r^5\Phi'(r) |\zeta_r|^2  dr 
\end{split}
\]
Hence we obtain 
\[
\begin{split}
&\int_0^R w^{2+\alpha}r^4 |\zeta_{rr}|^2  dr+4\int _0^Rw^{2+\alpha}r^2 |\zeta_r|^2 dr=  \frac{(\ast)}{\gamma^2}
\\ &\quad  + \underbrace{3\int_0^R w^{1+\alpha} r^4\Phi(r)|\zeta_r|^2 dr-4\frac{2+\alpha}{1+\alpha}\int_0^R w^{1+\alpha} r^4\Phi(r)|\zeta_r|^2 dr -\int _0^Rw^{1+\alpha}r^5\Phi'(r) |\zeta_r|^2  dr }_{\precsim \mathcal{E}^0}
\end{split}
\]
It is clear that the last three terms in the right-hand-side are bounded by the zeroth-order energy $\mathcal{E}^0$. 
Now combining it with \eqref{49}, \eqref{i}, and \eqref{ii}, we get 
\[
\mathcal{E}^{1,1}_r\leq C_0\theta_1^2\mathcal{E}^{1,1}_r + C_1(\mathcal{E}^0+\mathcal{E}^1).  
\]
 For sufficiently small $\theta_1>0$, by absorbing the first term in the right-hand-side into the left-hand-side, we deduce that 
\[
\mathcal{E}^{1,1}_r\leq C_2(\mathcal{E}^0+\mathcal{E}^1) 
\]
for some constant $C_2>0$. This finishes the proof for the case of $j=1$ and $k=1$. 
\end{proof}

We now turn into the cases $j\geq 2$. As in the case of $j=1$, we will directly use the equation and take advantage of the elliptic estimates. What is subtle here is to capture the correct behavior of solutions in the normal direction $\partial_r$ near the boundary. 

\begin{lemma}[$\mathcal{E}^{j,k};1\leq k\leq j$, $2\leq j\leq \lceil\alpha\rceil +3$]\label{lem4.9} 
 There exists a constant  $C>0$ such that 
\[
\mathcal{E}^{j,k}\leq C\sum_{l=0}^j \mathcal{E}^l(t) . 
\]
\end{lemma}

\begin{proof}  Notice that because of \eqref{Ej0}, it suffices to show that each spatial energy term $\mathcal{E}_r^{j,k}$ for $1\leq k\leq j$ is bounded by $\mathcal{E}(t)$. We will present the detail for $j=2$; other cases follow by the induction on $j,k$. When $k=1$, the spatial energy term $\mathcal{E}^{2,1}_r$ contains one temporal derivative and two spatial derivatives. The time derivative of the equation \eqref{zetaE2} is a good place to start. But then, the time derivative does not affect the weights at all since $w$ and $r$ do not change in time. 
Therefore,  following the same procedure for $\mathcal{E}^{1,1}_r$ in the previous lemma, we can deduce that 
\[
\mathcal{E}^{2,1}_r\leq C(\mathcal{E}^0+\mathcal{E}^1+\mathcal{E}^2)
\]
for some constant $C>0$.

To deal with $\mathcal{E}^{2,2}_r$ which contains three spatial derivatives, we will first derive the equation for $\zeta_{rrr}$ from \eqref{zetaE2}. Here is the algorithm to do so: first divide both sides of \eqref{zetaE2} by $r^3w^\alpha$: 
\begin{equation*}
\begin{split}
 \gamma \big \{w r\zeta_{rr} +(1+\alpha) w_r r\zeta_r+4w\zeta_r\big\}
& = \frac{ r \zeta_{tt}}{(1+\zeta)^2} +(3\gamma-4) r\Phi(r) \zeta- r\Phi(r) f(\zeta) \\
&\quad +\big\{ w r\partial_2h\zeta_{rr} +w(\partial_1h+\partial_2h)\zeta_r+(1+\alpha) w_r h \big\}
\end{split}
\end{equation*}
Then we take $\partial_r$ of both sides of the above equation and move the terms involving $\zeta_r$ into the right-hand-side to get 
\begin{equation}\label{zetaE3}
\begin{split}
 \gamma \big \{
w r\zeta_{rrr} +(2+\alpha) w_r r\zeta_{rr}+5w\zeta_{rr}\big\} = -\gamma\big\{(5+\alpha)w_r \zeta_r +(1+\alpha)w_{rr}r\zeta_r\big\}
\\ + \frac{ r \zeta_{ttr}}{(1+\zeta)^2}+ \frac{ \zeta_{tt}}{(1+\zeta)^2}- \frac{ 2r\zeta_{tt}\zeta_r}{(1+\zeta)^3} +(3\gamma-4) \big(r\Phi(r) \zeta\big)_r- \big(r\Phi(r) f(\zeta)\big)_r \\
 \quad+ \big( w r\partial_2h\zeta_{rr} +w(\partial_1h+\partial_2h)\zeta_r+(1+\alpha) w_r h \big)_r
\end{split}
\end{equation}
As in the previous lemma, we square both sides of \eqref{zetaE3}, multiply by $w^{1+\alpha} r^2$ -- here the multiplier $w^{1+\alpha}$ has been  chosen inspired by the analysis carried out in \cite{JM10} -- and integrate it over $(0,R)$ to get 
\begin{equation}
\begin{split}\label{56}
&\underbrace{\int_0^R {w^{1+\alpha} r^2} \big| w r\zeta_{rrr} +(2+\alpha) w_r r\zeta_{rr}+5w\zeta_{rr}  \big |^2 dr}_{(\ast)} \\
&\quad\precsim \underbrace{\int_0^Rw^{1+\alpha} r^2 (|w_r|^2+ |rw_{rr}|^2) |\zeta_r |^2 dr}_{(i)}\\
&\quad\quad + \underbrace{\int_0^R  \frac{w^{1+\alpha} r^4|\zeta_{ttr}|^2}{(1+\zeta)^4} dr }_{\precsim \,\mathcal{E}^{2,0}_r} +\underbrace{\int_0^R  \frac{w^{1+\alpha} r^2|\zeta_{tt}|^2}{(1+\zeta)^4} dr + \int_0^R  \frac{w^{1+\alpha} r^2|\zeta_{tt}r\zeta_r|^2}{(1+\zeta)^6} dr }_{(ii)} \\
&\quad\quad+\underbrace{\int_0^Rw^{1+\alpha} r^2 |\big(r\Phi(r) \zeta\big)_r |^2dr +\int_0^Rw^{1+\alpha} r^2 |\big(r\Phi(r) f(\zeta)\big)_r|^2 dr}_{(iii)}  
\\
&\quad\quad +\underbrace{\int_0^Rw^{1+\alpha} r^2  \left| \big( w r\partial_2h\zeta_{rr} +w(\partial_1h+\partial_2h)\zeta_r+(1+\alpha) w_r h \big)_r \right |^2 dr}_{(iv)}
\end{split}
\end{equation}
Note that $(i)$, $(ii)$, $(iii)$ contain stronger weights near the origin and for those terms,  we can use the localized Hardy inequality \eqref{hardy0} to obtain 
\begin{equation}\label{i,ii,iii}
(i)\precsim \mathcal{E}^0 + \mathcal{E}^{1}, \quad (ii)\precsim \mathcal{E}^{1}+  \mathcal{E}^{2}, \quad 
(iii)\precsim \mathcal{E}^0+\mathcal{E}^1.
\end{equation}
The last term in the right-hand-side of \eqref{56}: $(iv)$ includes $\zeta_{rrr},\,\zeta_{rr},\,\zeta_r$ with different weights and it can be treated in a similar way as done for  $(iv)$ of \eqref{49} in the previous lemma. We expand it out and apply the Hardy inequalities both near the origin  \eqref{hardy0} and near the boundary  \eqref{Hardy-gw} to deduce that 
\begin{equation}\label{iv}
(iv)\precsim \theta_1^2\{ \mathcal{E}^0+\mathcal{E}^{1,1}_r+ \mathcal{E}^{2,2}_r \}. 
\end{equation}

What follows now is the elliptic estimate for the $(\ast)$ in the left-hand-side of \eqref{56}, which will give rise to the term $\mathcal{E}^{2,2}_r$. 
\[
\begin{split}
(\ast)&=\int_0^R w^{3+\alpha}r^4 |\zeta_{rrr}|^2  dr +(2+\alpha)^2\int_0^R  w^{1+\alpha}r^4|w_r|^2|\zeta_{rr}|^2  dr+ 25\int_0^R  w^{3+\alpha}r^2 |\zeta_{rr}|^2 dr\\
&\;\;\;+\underbrace{2(2+\alpha)\int_0^R  w^{2+ \alpha}r^4 w_r\zeta_{rrr}\zeta_{rr} dr}_{(\ast)_1} +\underbrace{  10\int_0^R  w^{3+\alpha} r^3 \zeta_{rrr}\zeta_{rr}dr}_{(\ast)_2}\\
&\;\;\;+{10(2+\alpha) \int_0^R  w^{2+\alpha} r^3w_r|\zeta_{rr}|^2 dr}
\end{split}
\]
For the $(\ast)_1$ and $(\ast)_2$, we integrate by parts to get 
\[
\begin{split}
\frac{(\ast)_1}{2+\alpha}&=-\int_0^R  (w^{2+\alpha})_r r^4w_r |\zeta_{rr}|^2 dr -4\int_0^R  w^{2+\alpha} r^3w_r|\zeta_{rr}|^2 dr -\int_0^R  w^{2+\alpha} r^4 w_{rr}|\zeta_{rr}|^2 dr \\
&= -(2+\alpha)\int _0^R w^{1+\alpha} r^4|w_r|^2|\zeta_{rr}|^2 dr-4\int _0^R w^{2+\alpha} r^3w_r|\zeta_{rr}|^2 dr \\
&\quad-\int_0^R  w^{2+\alpha} r^4 w_{rr}|\zeta_{rr}|^2 dr \\
(\ast)_2&=-5(3+\alpha)\int_0^R  w^{2+\alpha} w_r r^3|\zeta_{rr}|^2 dr - 15 \int_0^R  w^{3+\alpha} r^2 |\zeta_{rr}|^2 dr 
\end{split}
\]
Thus we obtain 
\[
\begin{split}
&\int_0^R  w^{3+\alpha}r^4 |\zeta_{rrr}|^2  dr + 10\int_0^R  w^{3+\alpha}r^2 |\zeta_{rr}|^2 dr \\
& \quad=(\ast)-(\alpha-3)\int_0^R  w^{2+\alpha} w_r r^3|\zeta_{rr}|^2 dr +(2+\alpha)\int_0^R  w^{2+\alpha} r^4 w_{rr}|\zeta_{rr}|^2 dr 
\end{split}
\]
By noting $w_r=-r\Phi(r)/(1+\alpha)$, we see that the last two terms are bounded by $\mathcal{E}^{1,1}_r$.  By combining with  \eqref{56}, \eqref{i,ii,iii} and \eqref{iv}, we deduce that 
\[
\mathcal{E}^{2,2}_r\leq C(\mathcal{E}^0+\mathcal{E}^1+\mathcal{E}^2+\mathcal{E}^{1,1}_r) 
\]
for some constant $C>0$. Since $\mathcal{E}^{1,1}_r\precsim \mathcal{E}^0+\mathcal{E}^1$ by the previous lemma, the desired result  follows and this finishes the proof of the case $j=2$. Other cases can be done inductively: take $\partial_r$ derivatives of the equation \eqref{zetaE3}, square it, multiply  by appropriate weights depending on the number of spatial derivatives, and exploit the Hardy inequalities and the elliptic estimates. The procedure and the estimates are similar to the previous cases and we omit the details.  
\end{proof}



\section{Nonlinear weighted energy estimates}\label{5}

In this section, we develop the nonlinear energy estimates for sufficiently small perturbed solutions $(\zeta_t,\zeta)$ to the Euler-Poisson system \eqref{zeta} or  \eqref{zetaE}. Because of a subtle nonlinear structure of the pressure gradient term, a splitting to a linear part and a nonlinear part as done in \eqref{zetaE} may destroy the important structure during the estimates and it may lose a necessary cancelation property unless one would be extremely cautious. Here in order to obtain the close energy inequalities, we will work with the original perturbed form \eqref{zeta} and introduce the nonlinear high-order energy norms in \eqref{varE}.  We first recast \eqref{zeta} in the following form: 
\begin{equation}\label{Ezeta}
\begin{split}
\frac{w^\alpha r^4 \zeta_{tt}}{(1+\zeta)^2} + r^3\partial_r\left(w^{1+\alpha} \left\{ \Big[
 1+ \frac{1}{r^2} \Big( r^3 (\zeta+\zeta^2+\frac{\zeta^3}{3}) \Big)_r
  \Big]^{-\frac{1+\alpha}{\alpha}}   -1\right\} \right) &\\
  +  \frac{ 1- (1+\zeta)^4 }{(1+\zeta)^4 } w^\alpha r^4 \Phi(r)
&=0
\end{split}
\end{equation}
where we have used the Lane-Emden equation \eqref{EQ} with the notation introduced in \eqref{Phi} and the following identity: 
\[
\begin{split}
&(1+\zeta)^2(1+\zeta+\zeta_rr) = (1+2\zeta+\zeta^2)(1+\zeta+\zeta_r r)\\
&\quad\quad= 1+ 3\zeta + \zeta_r r+   3\zeta^2 + 2\zeta \zeta_r r + \zeta^3+\zeta^2\zeta_r r\\
&\quad\quad= 1+ \frac{1}{r^2} \Big( r^3 (\zeta+\zeta^2+\frac{\zeta^3}{3}) \Big)_r
\end{split}
\]

We are now ready to perform the energy estimates. Throughout the section, the smallness of the solution \eqref{small} is assumed. 
We start with $\mathcal{E}^0$.

\begin{lemma}[$\mathcal{E}^0$] For any fixed small $\kappa>0$
\begin{equation}\label{E_0}
\frac12\frac{d}{dt}\mathcal{E}^0 \leq \frac{d}{dt}K_0+(C\theta_1+\kappa) \mathcal{E}^0
+C_\kappa \int_0^R w^\alpha r^4 \zeta^2 dr
\end{equation}
where $|K_0|\leq C\theta_1 \mathcal{E}^0$. 
\end{lemma}

\begin{proof} We begin by multiplying \eqref{Ezeta} by $(1+\zeta)^2\zeta_t $ and integrating over $(0,R)$: 
\[
\begin{split}
\int_0^R  {w^\alpha r^4\zeta_{tt}\zeta_t}  dr + \int_0^R (1+\zeta)^2 \zeta_t \,
 r^3\partial_r\left(w^{1+\alpha} \Big\{ \Big[
 1+ \frac{1}{r^2} \Big( r^3 (\zeta+\zeta^2+\tfrac{\zeta^3}{3}) \Big)_r
  \Big]^{-\frac{1+\alpha}{\alpha}}   -1\Big\} \right) dr  \\
  + \int_0^R  \frac{ [1- (1+\zeta)^4]\zeta_t }{(1+\zeta)^2 } w^\alpha r^4 \Phi(r) dr =0
\end{split}
\]
We denote the left-hand-side by $(I)+(II)+(III)$. We will estimate it term by term. Note that the first term $(I)$ forms a perfect time derivative:  
\[
(I)=\frac{1}{2}\frac{d}{dt}\int_0^R w^\alpha r^4| \zeta_t |^2 dr . 
\]
For $(II)$, we first integrate by parts
\[
\begin{split}
(II)& = -  \int _0^R \big(r^3(1+\zeta)^2 \zeta_t \big)_r \,w^{1+\alpha} \Big\{ \Big[
 1+ \frac{1}{r^2} \Big( r^3 (\zeta+\zeta^2+\tfrac{\zeta^3}{3}) \Big)_r
  \Big]^{-\frac{1+\alpha}{\alpha}}   -1\Big\}  dr 
 \end{split}
\]
and use the identity $(1+\zeta)^2 \zeta_t=(\zeta+\zeta^2+\frac{\zeta^3}{3})_t$ to see that it also forms a perfect time derivative 
 \[
\begin{split}
(II) & = -\int _0^R \big(r^3(\zeta+\zeta^2+\frac{\zeta^3}{3})_t \big)_r \,w^{1+\alpha} \Big\{ \Big[
 1+ \frac{1}{r^2} \Big( r^3 (\zeta+\zeta^2+\tfrac{\zeta^3}{3}) \Big)_r
  \Big]^{-\frac{1+\alpha}{\alpha}}   -1\Big\}  dr \\
  &= \frac{d}{dt}\int _0^R w^{1+\alpha}  \Big\{  \alpha r^2  \Big( \Big[
 1+ \frac{1}{r^2} \Big( r^3 (\zeta+\zeta^2+\tfrac{\zeta^3}{3}) \Big)_r
  \Big]^{-\frac{1}{\alpha}} -1 \Big)+  \Big( r^3 (\zeta+\zeta^2+\tfrac{\zeta^3}{3}) \Big)_r  \Big\}   dr
\end{split}
\]
By Taylor's theorem, we see that the inside of $\{\cdot\}$ is non-negative and it can be written as 
\[
\begin{split}
\frac12(1+\frac{1}{\alpha}) \frac{1}{r^2} \left|  \Big( r^3 (\zeta+\zeta^2+\tfrac{\zeta^3}{3}) \Big)_r \right|^2 + \mathfrak{h}=\frac12(1+\frac{1}{\alpha}) \big|3r\zeta+r^2 \zeta_r \big|^2 + \mathfrak{h}
\end{split}
\]
where $\mathfrak{h}$ together with the weight $w^{1+\alpha}$ is bounded by $\theta_1\mathcal{E}^0$. 
Thus we may write 
\[
(II) =\frac12(1+\frac{1}{\alpha}) \frac{d}{dt} \underbrace{\int_0^Rw^{1+\alpha}   \big|3r\zeta+r^2 \zeta_r \big|^2dr}_{(\ast)}  - \frac{d}{dt}K_0 
\]
where $K_0\leq C\theta_1\mathcal{E}^0$.  The term $(\ast)$ gives rises to the other part of $\mathcal{E}^0$. To see it, we expand it out: 
\[
\begin{split}
\int_0^R w^{1+\alpha}   \big|3r\zeta+r^2 \zeta_r \big|^2 dr = 9\int_0^R w^{1+\alpha} r^2  |\zeta|^2 dr + \int_0^R w^{1+\alpha}  r^4 | \zeta_r |^2 dr +6 \int _0^R w^{1+\alpha}  r^3 \zeta \zeta_r dr 
\end{split}
\]
For the last term in the right-hand-side, by integrating by parts we obtain
\[
6 \int _0^Rw^{1+\alpha}  r^3 \zeta \zeta_r dr = - 6\int_0^R w^{1+\alpha} r^2|\zeta|^2 dr - 3\int _0^R(w^{1+\alpha})_r r^3 |\zeta|^2 dr
\]
and in turn
\[
\int_0^R w^{1+\alpha}   \big|3r\zeta+r^2 \zeta_r \big|^2 dr= 
 \int_0^R w^{1+\alpha}  r^4 | \zeta_r |^2 dr  + 3\int_0^R w^{1+\alpha} r^2  |\zeta|^2 dr - 3\int_0^R (w^{1+\alpha})_r r^3 |\zeta|^2 dr. 
\]
Notice that all three terms are positive. 

Lastly, we use the Cauchy-Schwartz inequality to control $(III)$. Since $\Phi(r)$ is bounded and by using \eqref{small}, it is easy to deduce that 
\[
|(III)|\leq {\kappa} \int_0^R w^\alpha r^4 |\zeta_t|^2 dr + C_\kappa \int_0^R w^\alpha r^4 \zeta^2 dr . 
\]
 This establishes \eqref{E_0}. 
\end{proof}

For higher-order energy estimates, we will look at the weighted quantity instead of dealing with $\zeta_t$ and $\zeta_r$ directly: this way, we will have more effective way of getting the estimates. In fact, notice that the nonlinearity in \eqref{Ezeta} is directly related to $J$ defined in \eqref{J}: 
\[
1+ \frac{1}{r^2} \Big( r^3 (\zeta+\zeta^2+\frac{\zeta^3}{3}) \Big)_r =(1+\zeta)^2(1+\zeta+\zeta_rr) = \xi^2(\xi+\xi_r r)=J
\]
and hence, the equation  \eqref{Ezeta} can be written as 
\begin{equation}\label{Jzeta}
\frac{w^\alpha r^4 \zeta_{tt}}{(1+\zeta)^2} + r^3\partial_r\big( w^{1+\alpha}\{J^{-\frac{1+\alpha}{\alpha}} -1\}  \big) +\frac{ 1- (1+\zeta)^4 }{(1+\zeta)^4 } w^\alpha r^4 \Phi(r)=0 . 
\end{equation}
Moreover, we observe that 
\[
J_t=\frac{1}{r^2} \big[r^3 (1+\zeta)^2\zeta_t \big]_r\,. 
\]
This motivates us to consider the following variable: 
\begin{equation}\label{varphi}
\varphi\equiv (1+\zeta)^2\zeta_t \; \;\text{   so   that }\; \;J_t=\frac{1}{r^2} [r^3\varphi]_r \,.
\end{equation}
The time derivative of \eqref{Ezeta} reads as 
\begin{equation}\label{zetaEt}
\begin{split}
&\frac{w^\alpha r^4\zeta_{ttt}}{(1+\zeta)^2}-2\frac{w^\alpha r^4\zeta_{tt}\zeta_t}{(1+\zeta)^3} - 4\frac{ w^\alpha r^4 \Phi(r) \zeta_t }{(1+\zeta)^5}\\ 
&-\tfrac{1+\alpha}{\alpha}\, r^3\partial_r\left(w^{1+\alpha}  \big[
 1+ \tfrac{1}{r^2} ( r^3 (\zeta+\zeta^2+\tfrac{\zeta^3}{3}))_r
  \big]^{-\frac{1+2\alpha}{\alpha}}  \cdot  \tfrac{1}{r^2} \big( r^3(1+\zeta)^2 \zeta_t \big)_r \right)=0 \\
\end{split}
\end{equation}
By using the definition of $\varphi$, we also obtain 
\[
\varphi_t=(1+\zeta)^2\zeta_{tt}+2(1+\zeta)\zeta_t^2 \quad\text{and}\quad \varphi_{tt}=(1+\zeta)^2\zeta_{ttt}+4(1+\zeta)\zeta_t\zeta_{tt} +2\zeta_t^3
\]
and hence by a straightforward computation, one can see that \eqref{zetaEt} reads in terms of $\varphi$ as follows: 
\begin{equation}\label{Jtzeta}
\begin{split}
\frac{w^\alpha r^4 \varphi_{tt}}{(1+\zeta)^4} -\frac{6w^\alpha r^4\varphi \varphi_{t}}{(1+\zeta)^7}-\frac{14 w^\alpha r^4 \varphi^3}{(1+\zeta)^{10}} -\frac{4  w^\alpha r^4 \Phi(r) \varphi }{(1+\zeta)^7}&\\
 -\frac{1+\alpha}{\alpha} r^3\partial_r\left( w^{1+\alpha}J^{-\frac{1+2\alpha}{\alpha}} \frac{1}{r^2} [r^3\varphi]_r  \right)&=0 . 
\end{split}
\end{equation}
The advantage of the above $\varphi$ equation \eqref{Jtzeta} is that the last term, which contains full spatial derivative, is self-adjoint  and more or less linear with respect to $\varphi$ up to lower order terms, and that  such a structure will not be destroyed under the time differentiations. This makes the higher order temporal energy estimates affordable. 

We now introduce the nonlinear $\varphi-$energy $\mathfrak{E}^i$ for $i\geq 1$, 
\begin{equation}\label{varE}
\mathfrak{E}^i= \int _0^R\frac{w^\alpha r^4 \big| \partial_t^{i-1}\varphi_t \big|^2}{(1+\zeta)^4} dr 
+\frac{1+\alpha}{\alpha}\int_0^R w^{1+\alpha} J^{-\frac{1+2\alpha}{\alpha}} \frac{ 1 }{r^2} \left| (r^3\partial_t^{i-1}\varphi )_r \right|^2 dr 
\end{equation}
When $i\geq 1$, $\mathfrak{E}^i$ corresponds to the homogeneous weighted energy $\mathcal{E}^i$ and we will show in Lemma \ref{equivEE} at the end of this section that they are equivalent under the assumption \eqref{small}.

We record the high order energy inequalities for the solutions to \eqref{Jtzeta}. 

\begin{lemma}[$\mathfrak{E}^i;1\leq  i\leq \lceil\alpha\rceil+3$]\label{lem5.2} Under the assumption \eqref{small}, the solutions to \eqref{Jtzeta} satisfy the following: for any small fixed $\kappa>0$
\begin{equation*}
\frac12\frac{d}{dt}\mathfrak{E}^i \leq (C\theta_1+\kappa) \mathfrak {E}^i +C\mathcal{E}^0+ C\sum_{j=1}^{i-1}\mathfrak{E}^j. 
\end{equation*}
\end{lemma}

\begin{proof}
We start with $i=1$. We multiply \eqref{Jtzeta} by $\varphi_t$ and integrate it over $(0,R)$. We denote each integral by $I_k$ for $1\leq k\leq 5$. We will estimate them term by term. $I_1$ forms an energy plus a commutator and thus $I_1+I_2$ can be written as  
\[
\begin{split}
I_1+I_2&= \frac{1}{2}\frac{d}{dt}\int_0^R \frac{ w^\alpha r^4 |\varphi_{t}|^2}{(1+\zeta)^4} dr - \int_0^R  \frac{4w^\alpha r^4\zeta_{t}(\varphi_t)^2}{(1+\zeta)^5} dr
\end{split}
\]
where we have used \eqref{varphi} for $I_2$. Note that the second term is bounded by $\theta_1\mathfrak{E}^1$ since $|\zeta_t|\leq \theta_1$ due to \eqref{small}. It is easy to see that $|I_3|\leq C\theta_1^2(\mathfrak{E}^1+\mathcal{E}^0)$. For $I_4$,  $\Phi(r)$ is bounded and by using the Cauchy-Swartz inequality, we see that 
\[
|I_4|\leq \kappa\, \mathfrak{E}^1 +C_\kappa \mathcal{E}^0. 
\]
We next move onto $I_5$. We first integrate it by parts in $r$ and then we see that it forms a perfect time derivative plus a commutator: 
\[
\begin{split}
I_5&= \frac{1+\alpha}{\alpha}\int_0^R  ( r^3\varphi_t)_r\cdot w^{1+\alpha}J^{-\frac{1+2\alpha}{\alpha}} \frac{1}{r^2} (r^3\varphi)_r   dr  \\
  &= \frac12 \frac{1+\alpha}{\alpha} \frac{d}{dt} \int_0^R w^{1+\alpha} J^{-\frac{1+2\alpha}{\alpha}} \frac{ 1 }{r^2} \left| (r^3\varphi )_r \right|^2dr \\
  &\quad+   \frac12 \frac{1+\alpha}{\alpha}   \frac{1+2\alpha}{\alpha} 
  \int_0^R w^{1+\alpha} J^{-\frac{1+3\alpha}{\alpha}}J_t \frac{ 1 }{r^2} \left| (r^3\varphi )_r \right|^2dr
\end{split}
\]
Since $J$ is bounded away from zero and from above and $J_t$ is bounded by $\theta_1$ because of \eqref{small}, the second term is bounded by $\theta_1\mathfrak{E}^1$. This completes the case of $i=1$. 

For higher order estimates of $\mathfrak{E}^i$ where $2\leq i\leq \lceil \alpha\rceil +3$, we first write out the $(i-1)^{th}$ temporal derivative of \eqref{Jtzeta}:
\begin{equation}\label{Jtizeta}
\begin{split}
&\frac{w^\alpha r^4 \partial_t^{i-1} \varphi_{tt}}{(1+\zeta)^4} +  
\sum_{j=1}^{i-1}c_{1j} w^\alpha r^4 \partial_t^{i-1-j} \varphi_{tt}\partial_t^{j}\left[ \frac{1}{(1+\zeta)^4}\right]   -\partial_t^{i-1}\left[ \frac{6w^\alpha r^4\varphi \varphi_{t}}{(1+\zeta)^7}\right]  \\& -  \partial_t^{i-1}\left[\frac{14 w^\alpha r^4 \varphi^3}{(1+\zeta)^{10}}\right]  - \partial_t^{i-1}\left[ \frac{4  w^\alpha r^4 \Phi(r) \varphi }{(1+\zeta)^7} \right]  -\tfrac{1+\alpha}{\alpha} r^3\partial_r\Big( w^{1+\alpha}J^{-\frac{1+2\alpha}{\alpha}} \frac{1}{r^2} [r^3 \partial_t^{i-1}\varphi]_r  \Big) \\
&+ \sum_{j=1}^{i-1}c_{2j} r^3\partial_r\Big( w^{1+\alpha}\partial_t^j[J^{-\frac{1+2\alpha}{\alpha}}] \frac{1}{r^2} [r^3\partial_t^{i-1-j}\varphi]_r  \Big) =0 
\end{split}
\end{equation}
where $c_{1j}$ and $c_{2j}$ are binomial coefficients. 
Now we multiply \eqref{Jtizeta} by $\partial_t^i\varphi$ and integrate it over $(0,R)$. We denote each integral 
by $J_k$ for $1\leq k\leq 7$. As before, we will estimate them term by term. As in the case of $I_1$, the first term  $J_1$ forms an energy plus a commutator that is bounded by $\theta_1\mathfrak{E}^i$: 
\[
\begin{split}
J_1= \frac{1}{2}\frac{d}{dt}\int_0^R \frac{ w^\alpha r^4 |\partial_t^i\varphi|^2}{(1+\zeta)^4} dr + \int_0^R  \frac{2w^\alpha r^4\zeta_{t}(\partial_t^i\varphi)^2}{(1+\zeta)^5} dr
\end{split}
\]
Next, $J_2,J_3,J_4$ will yield lower-order nonlinear terms. Here we present the detail for $J_2$. We may assume $1\leq j\leq [\frac{i}{2}]+1$ since the other cases can be treated in the same way. We treat the case when $\partial_t$ hits $\zeta$ $j$ times for the second factor $\partial_t^{j}[\frac{1}{(1+\zeta)^2}]$. By Cauchy-Schwartz inequality, we first get 
\begin{equation}
\begin{split}
|J_2|\leq C  \underbrace{\sup|w^\frac{j-1}{2}\partial_t^j\zeta| }_{(\ast)_1}\underbrace{\left( \int_0^R \frac{{w^{\alpha-j+1} r^4|\partial_t^{i-j+1}\varphi|^2}}{(1+\zeta)^4} dr\right)^{\frac12}}_{(\ast)_2}  \left( \int_0^R \frac{w^\alpha r^4|\partial_t^i\varphi|^2}{(1+\zeta)^4} dr\right)^{\frac12}.\\
\end{split}\label{J2}
\end{equation}
Note that $(\ast)_1\leq \theta_1$ by the assumption \eqref{small}. For $(\ast)_2$, by Lemma \ref{equivEE}, it suffices to consider 
$$\widetilde{({\ast})_2} \equiv \left(\int_0^R {w^{\alpha-j+1} r^4|\partial_t^{i-j+2}\zeta|^2} dr\right)^{\frac12}$$
Then by the embedding inequality \eqref{6000} in Lemma \ref{lem4.5}, we get 
\[
[\widetilde{(\ast)_2}]^2\leq \sum_{k=0}^{j-1} \int_0^R w^{\alpha-j+1} w^{j-1+k} r^4 |\partial_t^{i-j+2}\partial_r^k\zeta|^2 dr\leq 
\sum_{k=0}^{j-1}\mathcal{E}^{i-j+1+k,k}\leq C\sum_{k=0}^{j-1}\mathcal{E}^{i-j+1+k}
\]
and hence, we deduce that 
\[
|J_2|\leq C\theta_1\left(\sum_{k=0}^{j-1}\mathcal{E}^{i-j+1+k}\right)^{\frac12} \left(\mathcal{E}^{i}\right)^{\frac12}\leq \theta_1\mathcal{E}^{i} + C\theta_1\sum_{k=0}^{j-1}\mathcal{E}^{i-j+1+k}. 
\]
The $J_5$ is also lower-order. When $\partial_t$ hits $\varphi$ $(i-1)$ times, we use the Cauchy-Schwartz inequality to derive that it is bounded by 
\[
  \frac{\kappa}{2} \int_0^R \frac{w^\alpha r^4|\partial_t^{i}\varphi|^2}{(1+\zeta)^4} dr  + C_\kappa \int_0^R \frac{w^\alpha r^4 |\partial_t^{i-1}\varphi|^2}{(1+\zeta)^4}dr
\]
and for other cases, we can use the above standard nonlinear estimates to conclude that 
\[
|J_5|\leq (\kappa+C\theta_1 )\mathfrak{E}^i + C_\kappa \mathfrak{E}^{i-1} + C\theta_1\sum_{k=0}^{i-2}\mathcal{E}^{ k}. 
\]
As in the previous  $I_5$ case,  the $J_6$ term will form an energy: 
\[
\begin{split}
J_6&= \frac{1+\alpha}{\alpha}\int_0^R  ( r^3\partial_t^i\varphi)_r\cdot w^{1+\alpha}J^{-\frac{1+2\alpha}{\alpha}} \frac{1}{r^2} (r^3\partial_t^{i-1}\varphi)_r   dr  \\
  &= \frac12 \frac{1+\alpha}{\alpha} \frac{d}{dt} \int_0^R w^{1+\alpha} J^{-\frac{1+2\alpha}{\alpha}} \frac{ 1 }{r^2} \left| (r^3\partial_t^{i-1}\varphi )_r \right|^2dr\\
  &\quad +   \frac12 \frac{1+\alpha}{\alpha}   \frac{1+2\alpha}{\alpha} 
  \int_0^R w^{1+\alpha} J^{-\frac{1+3\alpha}{\alpha}}J_t \frac{ 1 }{r^2} \left| (r^3\partial_t^{i-1}\varphi )_r \right|^2dr
\end{split}
\]
It is clear that the second term is bounded by $\theta_1\mathfrak{E}^i$. It now remains to estimate $J_7$. Like the $J_2$ term, this is a lower order nonlinear term. We may assume $1\leq j\leq [\frac{i}{2}]$ since the other cases can be treated in the same way. By Cauchy-Schwartz inequality, we first get 
\begin{equation}
\begin{split}
|J_7|\leq C  \underbrace{\sup|w^\frac{j-1}{2}\partial_t^j J | }_{(\ast)_3}&\underbrace{\left( \int_0^R  w^{2+\alpha-j }   \frac{ 1 }{r^2} \left| (r^3\partial_t^{i-1-j}\varphi )_r \right|^2dr \right)^{\frac12}}_{(\ast)_4}  \\
&\times\left( \int_0^R w^{1+\alpha}  J^{-\frac{1+2\alpha}{\alpha}} \frac{ 1 }{r^2} \left| (r^3\partial_t^{i-1}\varphi )_r \right|^2dr \right)^{\frac12}.\\
\end{split}\label{J7}
\end{equation}
Note that $(\ast)_3\leq \theta_1$ by the assumption \eqref{small}. For $(\ast)_4$, by Lemma \ref{equivEE}, it suffices to consider 
$$\widetilde{({\ast})_4} \equiv \left(\int_0^R {w^{2+\alpha-j} r^4|\partial_t^{i-j}\zeta_r|^2} dr\right)^{\frac12}$$
Then by the embedding inequality \eqref{6000} in Lemma \ref{lem4.5}
\[
[\widetilde{(\ast)_4}]^2\leq \sum_{k=0}^{j} \int_0^R w^{2+\alpha-j} w^{j+k} r^4 |\partial_t^{i-j}\partial_r^k\zeta_r|^2 dr\leq 
\sum_{k=0}^{j}\mathcal{E}^{i-j+k,k}\leq C\sum_{k=0}^{j}\mathcal{E}^{i-j+k}
\]
and hence, we deduce that 
\[
|J_7|\leq C\theta_1\left(\sum_{k=0}^{j}\mathcal{E}^{i-j+k}\right)^{\frac12} \left(\mathcal{E}^{i}\right)^{\frac12}\leq \theta_1\mathcal{E}^{i} + C\theta_1\sum_{k=0}^{j}\mathcal{E}^{i-j+k}. 
\]
\end{proof}

Next we establish the equivalence of energies $\mathfrak{E}$ and $\mathcal{E}$ under the assumption \eqref{small}.

\begin{lemma}\label{equivEE} Assume \eqref{small}. For each $1\leq i\leq \lceil \alpha \rceil+3$, there exists a constant $C>0$ such that 
\begin{equation}
  \mathfrak{E}^i = \mathcal{E}^i +\mathfrak{K}^i
\end{equation}
where $|\mathfrak{K}^i| \leq C\theta_1 \sum_{k=0}^i \mathcal{E}^k $. 
\end{lemma}

\begin{proof} This can be established by direct computation with the aid of Hardy inequality \eqref{hardy0} and Lemma \ref{lem4.9}. We will prove the case of $i=1$.  From \eqref{varphi} and \eqref{varE}, 
\[
\begin{split}
\mathfrak{E}^1&= \int_0^R \frac{w^\alpha r^4 \big|\varphi_t \big|^2}{(1+\zeta)^4} dr 
+\frac{1+\alpha}{\alpha}\int_0^R w^{1+\alpha} J^{-\frac{1+2\alpha}{\alpha}} \frac{ 1 }{r^2} \left| (r^3\varphi )_r \right|^2 dr \\
&=  \int_0^R \frac{w^\alpha r^4 \big|(1+\zeta)^2\zeta_{tt}+2(1+\zeta) \zeta_t^2\big|^2}{(1+\zeta)^4} dr 
\\
&\quad\quad+\frac{1+\alpha}{\alpha}\int_0^R w^{1+\alpha} J^{-\frac{1+2\alpha}{\alpha}} \frac{ 1 }{r^2} \left| (r^3 (1+\zeta)^2\zeta_t )_r \right|^2 dr\\
&\equiv (I)+(II)
\end{split}
\]
Now 
\[
\begin{split} 
(I)&= \int_0^R {w^\alpha r^4 \big|\zeta_{tt}+\frac{2\zeta_t^2}{1+\zeta}\big|^2}dr\\
& =
 \int_0^R {w^\alpha r^4 \big|\zeta_{tt}\big|^2}dr +{ \int_0^R { \frac{4w^\alpha r^4\zeta_t^2 \zeta_{tt} }{1+\zeta}}dr+ \int_0^R {\frac{4w^\alpha r^4\zeta_t^4}{(1+\zeta)^2}}dr}\equiv\mathcal{E}^1_t+{\mathfrak{K}^1_1}
\end{split}
\] 
and 
\[
\begin{split}  
  \frac{\alpha}{1+\alpha}(II)&=\int _0^Rw^{1+\alpha} r^4 J^{-\frac{1+2\alpha}{\alpha}} (1+\zeta)^4| \partial_t\zeta_r |^2 dr+ \int_0^R w^{1+\alpha} J^{-\frac{1+2\alpha}{\alpha}} \frac{ 1 }{r^2} \left| (r^3 (1+\zeta)^2)_r\zeta_t  \right|^2 dr\\
  &\quad+2 \int _0^Rw^{1+\alpha} J^{-\frac{1+2\alpha}{\alpha}}  r (1+\zeta)^2\zeta_{tr} (r^3 (1+\zeta)^2)_r\zeta_t   dr\\
  &\equiv \mathcal{E}^1_t+{\mathfrak{K}^1_2}
\end{split}
\] 
so that 
\[
\begin{split}
\mathfrak{K}^1_2 &=\int_0^R w^{1+\alpha} r^4 \big[J^{-\frac{1+2\alpha}{\alpha}} (1+\zeta)^4-1\big]| \partial_t\zeta_r |^2 dr+ \int_0^R w^{1+\alpha} J^{-\frac{1+2\alpha}{\alpha}} \frac{ 1 }{r^2} \left| (r^3 (1+\zeta)^2)_r\zeta_t  \right|^2 dr\\
  &\quad+2 \int_0^R w^{1+\alpha} J^{-\frac{1+2\alpha}{\alpha}}  r (1+\zeta)^2\zeta_{tr} (r^3 (1+\zeta)^2)_r\zeta_t   dr
\end{split}
\]
Hence, we can write $\mathfrak{E}^1=(I)+(II)=\mathcal{E}^1+{\mathfrak{K}^1_1} +  \tfrac{1+\alpha}{\alpha} {\mathfrak{K}^1_2} $. It is easy to see that $$|\mathfrak{K}^1_1|\precsim \theta_1( \mathcal{E}^1+\mathcal{E}^0).$$ For $\mathfrak{K}^1_2$, we denote three integrals by $(i)+(ii)+(iii)$. 
Notice that 
\[
\begin{split}
|(i)|&\precsim \theta_1\int_0^R w^{1+\alpha} r^4 | \partial_t\zeta_r |^2 dr\;\text{ since }\;\big|J^{-\frac{1+2\alpha}{\alpha}} (1+\zeta)^4-1 \big|\precsim\theta_1 \\
|(ii)|&\precsim \theta_1^2\left( \int_0^R w^{1+\alpha} r^4 | \zeta_r |^2 dr +  \int_0^R w^{1+\alpha} r^2 | \zeta |^2 dr \right)\text{ since } |\zeta_t|\leq \theta_1
\\&\precsim \theta_1^2\left( \int _0^R w^{1+\alpha} r^4 | \zeta_r |^2 dr +  \int_0^R w^{1+\alpha} r^4 | \zeta |^2 dr \right)\text{ by Hardy inequality }\eqref{hardy0}\\
|(iii)|&\precsim \theta_1\left(\int_0^R w^{1+\alpha} r^4 | \zeta_{tr} |^2 dr+ \int_0^R w^{1+\alpha} r^4 | \zeta_r |^2 dr +  \int_0^R w^{1+\alpha} r^4 | \zeta |^2 dr \right)
\end{split}
\]
where in order to obtain the estimate of $(iii)$, we have applied Cauchy-Swartz inequality and used $  |\zeta_t|\leq \theta_1$ and the estimate of $(ii)$.  Therefore, we conclude that $$|\mathfrak{K}^1_2|\precsim\theta_1 ( \mathcal{E}^1+\mathcal{E}^0)$$ and this finishes the proof of $i=1$. 

Other cases of $i\geq 2$ can be treated in the same manner by expanding $\partial_t^{i-1}\varphi$ in terms of $\zeta$, $\zeta_t$,...,$\partial_t^i\zeta$ in the energy form \eqref{varE}: since 
\[
\partial_t^{i-1}\varphi =(1+\zeta)^2\partial_t^i\zeta + \sum_{j=1}^{i-1} c_{j}\partial_t^{j}[(1+\zeta)^2]\,\partial_t^{i-j}\zeta
\]
where $c_j$'s are  binomial coefficients. By plugging into \eqref{varE}, we see that the leading-order term gives rise to $\mathcal{E}^i$. Following the same spirit of the above computations on $\mathfrak{K}^1_1$ and $\mathfrak{K}^1_2$ and exploiting the nonlinear estimates as done in the previous lemma by using \eqref{small}, Hardy inequality \eqref{hardy0} and Lemma \ref{lem4.9}, one can deduce that lower-order nonlinear terms are bounded by $\theta_1 \sum_{k=0}^i \mathcal{E}^k$. 
\end{proof}


\section{Nonlinear instability}\label{6}

Based on the nonlinear estimates in the previous section, we are now ready to show a bootstrap argument that allows us to control the growth of $\mathcal{E}(t)$ in terms of the linear growth rate $\sqrt{\mu_0}$, showed in Section \ref{3}. The idea is  to assume that the lowest order energy $\mathcal{E}^0$ starting from small initial data grows no faster than the linear growth rate; then the energy inequalities in the last section allow for a bootstrap argument that claims that all of $\mathcal{E}$ grows no faster than the linear growth rate.

\begin{proposition}\label{prop6.1}
Let $(\partial_t\zeta,\,\zeta)$ be a solution to the Euler-Poisson equation \eqref{zetaE} such that 
\begin{equation*}
\sqrt{\mathcal{E}}(0) \leq C_0\delta \text{ and }\sqrt{\mathcal{E}^0}(t)\leq C_0\delta e^{\sqrt{\mu_0} t}\text{ for }0\leq t\leq T
\end{equation*}
where $\mathcal{E}^0(t)$ and $\mathcal{E}(t)$ are as defined in \eqref{E0} and \eqref{TE}. Then there exist $C_\star>0$ and $\theta_\star>0$  such that if $ 0 \leq t \leq \min \{T,T(\delta,\theta_\star)\}$ then
\begin{equation*}
\sqrt{\widetilde{\mathcal{E}}}(t)\leq C_\star\delta e^{\sqrt{\mu_0} t}\leq C_\star\theta_\star
\end{equation*}
where we have written $T(\delta,\theta_\star)=\frac{1}{\sqrt{\mu_0}}\ln \frac{\theta_\star}{\delta}$. 
\end{proposition}

\begin{proof} We will employ a bootstrap argument using all of the nonlinear
energy estimates derived in Section \ref{5}. We now choose $\theta_1$ and $\kappa$ sufficiently
small (but fixed) in all of the estimates  in Lemma \ref{lem5.2} and \ref{equivEE} so that $C\theta_1 + \kappa \leq  \sqrt{\mu_0}/2$ and $C\theta_1\leq 1/8$ in all of the energy inequalities. As before, throughout this proof we will write $C$ for a generic constant; we write this in place of $C$ to distinguish the constants from those appearing in the nonlinear energy estimates. 

Define $T^\ast$ by 
\[
T^\ast\equiv\sup \{t : \sqrt{\widetilde{\mathcal{E}}}(s)\leq \theta_1 \text{ for } 0\leq s\leq t \}
\]
Then by Lemma \ref{lem4.7}, the assumption \eqref{small} is satisfied by the solution to  the Euler-Poisson equation \eqref{zetaE} for $0\leq t\leq \min\{T,T^\ast\}$. 

Let $0\leq t\leq \min\{T,T^\ast\}$ be given. From  Lemma \ref{lem5.2}, we have that for each $1\leq i\leq \lceil \alpha\rceil+3$
\begin{equation}\label{86}
\frac12\frac{d}{dt}\mathfrak{E}^i \leq \frac{\sqrt{\mu_0}}{2}\mathfrak {E}^i +C\mathcal{E}^0+ C\sum_{j=1}^{i-1}\mathfrak{E}^j
\end{equation}

We begin with $i=1$. Notice that for $i=1$, the last term in \eqref{86} is not present. Since $\mathcal{E}^0(t)\leq C_0^2\delta^2e^{2\sqrt{\mu_0}t}$ by the assumption, by using Gronwall inequality, we see that 
\[
\mathfrak{E}^1 \leq C \delta^2e^{2\sqrt{\mu_0}t}. 
\]
By Lemma \ref{equivEE}, we deduce that 
\[
\mathcal{E}^1 \leq C \delta^2e^{2\sqrt{\mu_0}t}. 
\]
For $2\leq i\leq \lceil \alpha\rceil+3$, inductively, we apply Gronwall inequality to \eqref{86} to obtain 
\[
\mathfrak{E}^i \leq C \delta^2e^{2\sqrt{\mu_0}t}
\]
and in turn by using Lemma \ref{equivEE}, we derive that 
\[
\mathcal{E}^i \leq C \delta^2e^{2\sqrt{\mu_0}t}
\]
for all $1\leq i\leq \lceil \alpha\rceil+3$. Moreover, by Lemma \ref{lem4.9}, we also deduce that for $0\leq t\leq \min\{T,T^\ast\}$, 
\begin{equation}\label{87}
\widetilde{\mathcal{E}}(t) \leq \widetilde{C} \delta^2e^{2\sqrt{\mu_0}t}
\end{equation}
for some constant $\widetilde{C}>0$ independent of $\delta$. We now choose $\theta_\star>0$ such that  $\widetilde{C}(\theta_\star)^2<(\theta_1)^2$. We consider the following two cases.

(i) $T(\delta,\theta_\star)\leq \min\{T,T^\ast\}$: In this case, the conclusion follows without any additional
work. 

(ii) $T(\delta,\theta_\star)> \min\{T,T^\ast\}$: We claim that it must hold that $T\leq T^\ast < T(\delta,\theta_\star)$  in
which case the conclusion directly follows. To prove the claim, we note that otherwise we would have $T^\ast< T < T(\delta,\theta_\star)$. Letting $t = T^\ast$ from \eqref{87}, we get
\[
\mathcal{E}(T^\ast) \leq\widetilde{C} \delta^2e^{2\sqrt{\mu_0}T^\ast} < \widetilde{C} \delta^2e^{2\sqrt{\mu_0}T(\delta,\theta_\star)}=\widetilde{C}(\theta_\star)^2  
\]
by the definition of $T(\delta,\theta_\star)$. However, this is impossible due to our choice of $\theta_\star$ since it would then contradict the definition of $T^\ast$. Since we then find our desired estimate in both cases, this concludes the proof of the proposition.
\end{proof}

In order to establish the nonlinear instability, we will take the growing mode profile $\phi_0$ constructed in Section \ref{3} as the initial data to the Euler-Poisson equation as in \eqref{initial}. By Lemma \ref{lem3.4}, we deduce that  all the spatial energies in $\widetilde{\mathcal{E}}$ for $\phi_0$ are finite. See also  Remark \ref{rem3.7}. Then, of course, the initial temporal total energy ${\mathcal{E}}(0)$ for such $\phi_0$ is obtained from the initial spatial energy through the equation. 

 We are now ready to prove our main result. 

\begin{theorem}\label{thm} There exist $\theta_0>0$, $C>0$, and $0<\delta_0<2\theta_0$  such that for any $0<\delta\leq \delta_0$, there exists a family of solutions $(\partial_t\zeta^\delta,\, \zeta^\delta)$ to  the Euler-Poisson system \eqref{zetaE}  so that 
\[
\sqrt{\mathcal{E}}(0)\leq C \delta \text{ but }\sup_{0\leq t\leq T^\delta}\sqrt{\mathcal{E}^0}(t)\geq \theta_0. 
\]
Here $T^\delta$ is given by $T^\delta=\frac{1}{\sqrt{\mu_0}}\ln\frac{2\theta_0}{\delta}$. 
\end{theorem}

\begin{proof} We write the equation \eqref{zetaE} as 
\begin{equation}\label{88}
 {w^\alpha r^4\zeta_{tt}} = L\zeta +N\zeta
\end{equation}
where
\[
\begin{split}
L\zeta&=\frac{1+\alpha}{\alpha} (w^{1+\alpha} r^4 \zeta_r)_r+(3\frac{1+\alpha}{\alpha} -4)(1+\alpha) w^\alpha r^3w_r \zeta \\
N\zeta
&=(2\zeta+\zeta^2)\left\{ \frac{1+\alpha}{\alpha} (w^{1+\alpha} r^4 \zeta_r)_r+(3\frac{1+\alpha}{\alpha} -4)(1+\alpha) w^\alpha r^3w_r \zeta \right\} \\
&\quad- (1+\zeta)^2r^3  \partial_r (w^{1+\alpha} h(\zeta,\zeta_rr))-\frac{(1+\zeta)^4-1-4\zeta(1+\zeta)^4}{(1+\zeta)^2} r^3\partial_r(w^{1+\alpha})
\end{split}
\]
where $h$ is given by \eqref{exp}. Recall that $L\zeta$ was introduced in \eqref{L}.  Let $\phi_0$ be a growing mode constructed in Proposition \ref{prop3.1}. Normalize $\phi_0$ such that 
\begin{equation}\label{nor}
(1+\mu_0)\|\phi_0\|^2_{X^\alpha} + \|\phi_0\|^2_{Y^\alpha}=1
\end{equation}
where $\|\cdot\|_{X^\alpha}$ and $\|\cdot\|_{Y^\alpha}$ are given in \eqref{XY}. 
Consider a family of initial data $(\zeta, \zeta_t)\big|_{t=0}= \delta(\phi_0,\sqrt{\mu_0}\phi_0).$
Then due to the regularity of the growing mode $\phi_0$ obtain in Lemma \ref{lem3.3} and \ref{lem3.4}, $\sqrt{\mathcal{E}}(0)\leq C_0\delta$ for some constant $C_0>0$.   Let $(\zeta^\delta, \zeta^\delta_t)$ be $\mathcal{E}-$solutions to the Euler-Poisson equation \eqref{88} with the given initial data 
\begin{equation}\label{initial}
(\zeta^\delta, \zeta^\delta_t)\Big|_{t=0}= \delta(\phi_0,\sqrt{\mu_0}\phi_0).
\end{equation}
Define $T$ by 
\[
T\equiv\sup\{\tau : {\mathcal{E}^0} (t) \leq 3 \delta^2 e^{2\sqrt{\mu_0}t} \text{ for }0\leq t\leq \tau\}. 
\]
Then by Proposition \ref{prop6.1}, there exist $C_\star>0$ and $\theta_\star>0$ such that for 
$0\leq t\leq \min\{T,T(\delta,\theta_\star)\}$, 
\begin{equation}\label{90}
\sqrt{\widetilde{\mathcal{E}}}(t)\leq C_\star\delta e^{\sqrt{\mu_0} t}\leq C_\star\theta_\star
\end{equation}

We now wish to write the solutions $\zeta^\delta$ as linear and nonlinear parts. To do so, let 
$\mathcal{L}(t)\binom{\zeta_1}{\zeta_2}$ be the solution operator for the linearized equation 
\[
 {\zeta_{tt}} -\frac{1}{w^\alpha r^4}L\zeta=0
\]
with the given data $\binom{\zeta}{\zeta_t}\big|_{t=0}=\binom{\zeta_1}{\zeta_2}$. Note that $\mathcal{L}(t)\binom{\zeta_1}{\zeta_2}$ is given by $\Psi(t)$ in \eqref{L1}. Then by Duhamel's principle, the solutions to the nonlinear Euler-Poisson equation \eqref{88} can be written as 
\[
 {\zeta^\delta(t)}= \mathcal{L}(t) \binom{\delta\phi_0}{\delta \sqrt{\mu} \phi_0} + 
\int_0^t \mathcal{L}(t-s) \binom{ 0 }{\frac{1}{w^\alpha r^4}N\zeta^\delta (s)} ds. 
\]
For notational convenience, we introduce the following norm: 
\[
\left\| \Psi \right\|_0^2\equiv \left\| \Psi \right\|^2_{X^\alpha} + \left\| \partial_t\Psi \right\|^2_{X^\alpha} +\left\| \Psi  \right\|^2_{Y^\alpha} 
\]
We remark that $\|\cdot\|_0^2$ corresponds to the zeroth-order energy $\mathcal{E}^0$. 
Now since $\phi_0$ is a largest growing mode, we see that 
\[
\mathcal{L}(t) \binom{\delta\phi_0}{\delta \sqrt{\mu} \phi_0} = \delta e^{\sqrt{\mu_0}t} {\phi_0}. 
\]
Then with \eqref{nor}, we obtain  that 
\begin{equation}\label{91}
\left\| \mathcal{L}(t) \binom{\delta\phi_0}{\delta\sqrt{\mu} \phi_0} \right\|_0 = \delta e^{\sqrt{\mu_0}t}. 
\end{equation}
For the nonlinear part, we use Lemma \ref{lem4.1} and \ref{lem3.7} to derive that 
\[
\begin{split}
\left\| \int_0^t \mathcal{L}(t-s) \binom{0}{ \frac{1}{w^\alpha r^4}N\zeta^\delta (s) }ds \right\|_0\leq C\int_0^t e^{\sqrt{\mu_0}(t-s)} 
\left\|\frac{1}{w^\alpha r^4} N\zeta^\delta (s)\right\|_{X^\alpha} ds
\end{split}
\]
for a constant $C>0$ only depending on $\mu_0$. 
Notice that $N
\zeta=O(|\zeta|^2+|\zeta_r|^2+|\zeta_{rr}|^2)$. We will show that  $\|\frac{1}{w^\alpha r^4} N\zeta^\delta (s)\|_{X^\alpha}$ is bounded by $(\delta e^{\sqrt{\mu_0}s})^2$. To do so, first we crudely expand $\frac{1}{w^\alpha r^4} N\zeta^\delta $:  
\[
\begin{split}
\Big|\frac{1}{w^\alpha r^4} N\zeta\Big|& = \Big|(2\zeta+\zeta^2)\left\{ \frac{1+\alpha}{\alpha}\frac{ (w^{1+\alpha} r^4 \zeta_r)_r}{w^\alpha r^4}+(3\frac{1+\alpha}{\alpha} -4)(1+\alpha)\frac{w_r}{r} \zeta \right\} \\
&\quad\;\;- (1+\zeta)^2\frac{\partial_r (w^{1+\alpha} h(\zeta,\zeta_rr))}{w^\alpha r}-\frac{(1+\zeta)^4-1-4\zeta(1+\zeta)^4}{(1+\zeta)^2} (1+\alpha)\frac{w_r}{r}\Big|
\end{split}
\]
and then use the triangle inequality to get 
\[
\Big|\frac{1}{w^\alpha r^4} N\zeta\Big| \leq C(|\zeta|+|\zeta|^2)\big(|w\zeta_{rr}|+|w_r\zeta_r|+|\frac{w \zeta_r}{r}|+|\frac{w_r \zeta}{r}|\big)+C\big( |\frac{w_r h}{r} +\frac{w\partial_rh }{r}|  \big). 
\]
Then by using the $L^\infty$ estimate in Lemma \ref{lem4.7} and by the property of $w$ in Lemma \ref{lem2.1}, we obtain  
\[
\begin{split}
\left\|\frac{1}{w^\alpha r^4} N\zeta^\delta (s)\right\|_{X^\alpha}^2&= \int_0^R  w^\alpha r^4   \Big|\frac{1}{w^\alpha r^4} N\zeta^\delta(s)\Big|^2 dr\\
&\leq C\widetilde{\mathcal{E}} (s)\Big\{ \int_0^R w^{\alpha+2} r^4 |\zeta^\delta_{rr}|^2 dr +\underbrace{\int_0^R w^\alpha r^4 |\zeta^\delta_r|^2 dr }_{(a_1)}+\underbrace{\int_0^R w^{\alpha+2} r^2 |\zeta^\delta_r|^2 dr}_{(b_1)} \\
&\quad\quad\quad\; + \int_0^R w^\alpha r^4|\zeta^\delta|^2 dr +\underbrace{\int_0^R w^\alpha r^4 |h^\delta|^2 dr}_{(a_2)}+ \underbrace{\int_0^R w^{\alpha+2}r^2 |\partial_rh^\delta|^2dr}_{(b_2)}\Big\}. 
\end{split}
\]
The first terms in the second and third lines have the right forms of the total energy $\widetilde{\mathcal{E}}$. For other terms, 
we apply the Hardy inequality near the boundary \eqref{Hardy-gw} to $(a_1)$ and $(a_2)$ and the Hardy 
inequality near the origin \eqref{hardy0} to $(b_1)$ and $(b_2)$ to deduce that 
\[
\left\|\frac{1}{w^\alpha r^4} N\zeta^\delta (s)\right\|_{X^\alpha}^2 \leq C(\widetilde{\mathcal{E}} (s))^2. 
\]
Now by \eqref{90}, we get 
\begin{equation}
\begin{split}\label{92}
\left\| \int_0^t \mathcal{L}(t-s) \binom{ \frac{1}{w^\alpha r^4}N\zeta^\delta (s) }{0} ds \right\|_0&\leq C\int_0^t e^{\sqrt{\mu_0}(t-s)} \widetilde{\mathcal{E}} (s) ds\\
&\leq C\delta^2 \int_0^t  e^{\sqrt{\mu_0}(t-s)} e^{2\sqrt{\mu_0}s } ds\\
&\leq C_1 (\delta e^{\sqrt{\mu_0} t})^2
\end{split}
\end{equation}
where $C_1$ is a constant independent of $\delta$. Now if necessary, fix $\theta_\star$ sufficiently small so that 
$C_1\theta_\star\leq1/2 $. 

Next, we claim that $T(\delta,\theta_\star)\leq T$. Suppose not: $T(\delta,\theta_\star)>T$. Then by \eqref{91} and \eqref{92}, we have 
\[
\begin{split}
\big\|{\zeta^\delta}\big\|_0(T) &\leq \left\| \mathcal{L}(t) \binom{\delta\sqrt{\mu}\phi_0}{\delta \phi_0} \right\|_0 (T)+ \left\|\int_0^t \mathcal{L}(t-s) \binom{ \frac{1}{w^\alpha r^4}N\zeta^\delta (s) }{0} ds \right\|_0 (T) \\
&\leq \delta e^{\sqrt{\mu_0}T} + C_1\theta_\star \delta e^{\sqrt{\mu_0} T} \leq \frac{3}{2} \delta e^{\sqrt{\mu_0}T}
\end{split}
\]
which would contradict the definition of $T$. With $T(\delta,\theta_\star)\leq T$ in hand, we now see that by using   \eqref{91} and \eqref{92} again, 
\[
\sqrt{\mathcal{E}^0}(T(\delta,\theta_\star))=\big\|{\zeta^\delta}\big\|_0(T(\delta,\theta_\star))\geq 
\delta e^{\sqrt{\mu_0}T(\delta,\theta_\star)}  - \frac12 \delta e^{\sqrt{\mu_0}T(\delta,\theta_\star)}= \frac12 \theta_\star. 
\]
Set $\theta_0\equiv \frac12 \theta_\star$ and $T^\delta\equiv T(\delta,\theta_\star) =T(\delta,2\theta_0)$.   This finishes the proof of the theorem. 
\end{proof}


\

\textbf{Acknowledgements.} This work was supported in part by NSF Grant DMS-0908007. Part of work was done during the author's visit to ICERM, Brown University in the fall of 2011 and she thanks the institute for the support. Also, she would like to thank \textsc{Yan Guo} for his encouragement and interest on this work.

\

\end{document}